\documentclass{amsart}

\usepackage{amssymb, esint, graphicx, verbatim}

\usepackage{geometry}

\newtheorem{theorem}{Theorem}[section]
\newtheorem{lemma}[theorem]{Lemma}
\newtheorem{proposition}[theorem]{Proposition}
\newtheorem{corollary}[theorem]{Corollary}

\theoremstyle{definition}
\newtheorem{remark}[theorem]{Remark}
\newtheorem{example}[theorem]{Example}
\newtheorem*{ack}{Acknowledgements}

\usepackage[colorlinks=true,urlcolor=blue, citecolor=red,linkcolor=blue,
linktocpage,pdfpagelabels, bookmarksnumbered,bookmarksopen]{hyperref}
\usepackage[hyperpageref]{backref}

\numberwithin{equation}{section}

\title{Beyond Cheeger's constant}


\date{\today}

\author[Brasco]{Lorenzo Brasco}
\address[L.\ Brasco]{Dipartimento di Matematica e Informatica
	\newline\indent
	Universit\`a degli Studi di Ferrara
	\newline\indent
	Via Machiavelli 35, 44121 Ferrara, Italy}
\email{lorenzo.brasco@unife.it}

\dedicatory{In memory of Rodolfo Boschi}

\begin{document}

\begin{abstract}
The Cheeger constant of an open set of the Euclidean space is defined by minimizing the ratio ``perimeter over volume'', among all its smooth compactly contained subsets. We consider a natural variant of this problem, where the volume of admissible sets is raised to any positive power. We show that for {\it sublinear} powers, all these generalized Cheeger constants are equivalent to the standard one, by means of a universal two-sided estimate. We also show that this equivalence breaks down for {\it superlinear} powers. In this case, some weird phenomena appear. We finally consider the case of convex planar sets and prove an existence result, under optimal assumptions.
\end{abstract}

\maketitle

\begin{center}
\begin{minipage}{10cm}
\small
\tableofcontents
\end{minipage}
\end{center}

\section{Introduction}

\subsection{Cheeger's constant}

We recall that the {\it Cheeger constant} of an open set $\Omega\subseteq\mathbb{R}^N$ is defined by the following constrained isoperimetric--type problem
\begin{equation}
\label{cheeger}
h_1(\Omega)=\inf\left\{\frac{\mathcal{H}^{N-1}(\partial E)}{|E|}\, :\, E\Subset\Omega \mbox{ open set with smooth boundary}\right\}.
\end{equation}
Here the symbol $|\cdot|$ stands for the $N-$dimensional Lebesgue measure, while $\mathcal{H}^{N-1}$ is the $(N-1)-$dimensional Hausdorff measure.
This constant and the minimization problem that goes with it have attracted an increasing interest along the years. We cite for example \cite{BoBo, BoPa, BuCaCo, CC, CCN, Erco, KF, KLR, LeNeSa, LePr} and \cite{Pa} for some studies on this topic.
Without any attempt of describing here the main known results or listing all the papers connected with this problem, we refer to the survey \cite{Leo} by Leonardi. There, the interested reader will find some motivations and some auxiliary problems leading to the study of this constant, together with an account of the main achievements on the minimization problem. 
\par
Apart for being interesting in itself, Cheeger's constant plays a particularly intriguing role
in Spectral Geometry. This was the original motivation to introduce it in \cite{cheeger} (see also \cite{Bu1}).
In order to explain this point, let us define the following sharp Poincar\'e constant
\[
\lambda(\Omega)=\inf_{u\in C^\infty_0(\Omega)} \left\{\int_\Omega |\nabla u|^2\,dx\, :\, \int_\Omega |u|^2\,dx=1\right\}.
\]
Whenever the space $W^{1,2}_0(\Omega)$ is compactly embedded in $L^2(\Omega)$, we know that this quantity gives the first eigenvalue of the Dirichlet-Laplacian on $\Omega$. More generally,
this coincides with the bottom of the spectrum of the Dirichlet-Laplacian on $\Omega$ (see for example \cite[Chapter 10, Section 1.1]{BiSo}).
Then, we have the following celebrated lower bound
\begin{equation}
\label{cheegerineq}
\left(\frac{h_1(\Omega)}{2}\right)^2\le \lambda(\Omega),
\end{equation}
usually called {\it Cheeger's inequality}, see for example \cite[Chapter 4, Section 2]{Maz} for a proof.
\par 
It is quite remarkable that \eqref{cheegerineq} gives a lower bound on the spectrum of the Dirichlet-Laplacian on an open set, which holds without any assumption on the open set and with a constant having an intrinsic geometric content. These features already explains quite neatly the interest gained by the Cheeger constant along the years.

\subsection{A variation on the theme}

From the mathematical point of view, in principle there is no reason to confine ourselves to consider the ratio ``perimeter/volume'' in the definition of $h_1$. One could for example consider different powers of the volume of the admissible sets.
In other words, for an exponent $0<q<N/(N-1)$, we could consider the following {\it generalized Cheeger constant}
\[
h_q(\Omega)=\inf\left\{\frac{\mathcal{H}^{N-1}(\partial E)}{|E|^\frac{1}{q}}\, :\, E\Subset\Omega \mbox{ open set with smooth boundary}\right\},
\]
associated to every open set $\Omega\subseteq\mathbb{R}^N$.
In the case $q=1$, we are back with the familiar Cheeger constant \eqref{cheeger}. We observe that the restrictions on the exponent $q$ are those making the constant $h_q$ non-trivial (see Remark \ref{rem:restrizioni} below).
\par
As one may expect, we have not been the first ones to consider this kind of generalization. We cite for example \cite{Av, BBF, FiMP} and \cite{FuMP}, where these constants have been considered. More recently, a systematic study of these constants has been started in the papers \cite{CS} and \cite{PS}.
\par
Nevertheless, it seems that many interesting questions on $h_q$ are still not fully understood. We wish to list some of them, here below: the present paper will then focus on answering (at least partially) some of them. 
\vskip.2cm\noindent
A first question which may arise is the following one: 
\begin{equation}
\label{Q1}
\mbox{\it can one still prove ``universal'' lower bounds like \eqref{cheegerineq}, with $h_q$ in place of $h_1$?}
\tag{\bf Q1}
\end{equation}
Another natural question, which is somehow connected to the previous one, reads as follows:
\begin{equation}
\label{Q2}
\mbox{\it is it possible to compare $h_1(\Omega)$ and $h_q(\Omega)$ with $q\not=1$?}
\tag{\bf Q2}
\end{equation}

Finally, let us consider one more question, which appears quite reasonable and interesting: this concerns the relation of $h_q(\Omega)$ with sharp functional inequalities.
In order to formulate the question, we recall that $h_1(\Omega)$ coincides with the following sharp Poincar\'e inequality
\[
\lambda_{1,1}(\Omega)=\inf_{u\in C^\infty_0(\Omega)} \left\{\int_\Omega |\nabla u|\,dx\, :\, \int_\Omega |u|\,dx=1\right\},
\]
see for example \cite[Theorem 2.1.3]{Maz}. Then, by introducing the more general quantity
\[
\lambda_{1,q}(\Omega)=\inf_{u\in C^\infty_0(\Omega)} \left\{\int_\Omega |\nabla u|\,dx\, :\, \int_\Omega |u|^q\,dx=1\right\},
\]
one could ask the following question:
\begin{equation}
\mbox{\it is it true that $\lambda_{1,q}(\Omega)=h_q(\Omega)$, for every $0<q<\dfrac{N}{N-1}$?}
\tag{\bf Q3}
\label{Q3}
\end{equation}
Actually, for $q>1$ we already know that the answer is {\it yes}, still by\footnote{In order to help the reader, we point out that one should use the result of \cite{Maz} with the following choices
\[
\mu(g)=|g|,\qquad \sigma(\partial g)=\mathcal{H}^{N-1}(\partial g),\qquad \beta=0,\qquad \Phi(x,\nabla u)=|\nabla u|,
\]
and $g\Subset\Omega$ has smooth boundary.} \cite[Theorem 2.1.3]{Maz}. One may wonder what happens for $0<q<1$.
\vskip.2cm\noindent
The previous questions were more focused on the constant $h_q(\Omega)$ itself. One could try to have a closer look at the minimization problem which defines it and study for example: the behaviour of minimizing sequences, existence of a (possibly relaxed) solution, its regularity and so on. Some interesting results of this type have been obtained by Pratelli and Saracco in \cite{PS}, we will comment below on these results.
\par
Let us first make a couple of preliminary observations:
\begin{enumerate}
\item if $E_1,E_2\Subset \Omega$ are disjoint open subsets with smooth boundary, then we have
\[
\mathcal{H}^{N-1}(\partial (E_1\cup E_2))=\mathcal{H}^{N-1}(\partial E_1)+\mathcal{H}^{N-1}(\partial E_2),
\]
while
\[
\left|E_1\cup E_2\right|^\frac{1}{q} \left\{\begin{array}{ll}
>|E_1|^\frac{1}{q}+|E_2|^\frac{1}{q}, & \mbox{ if } 0<q<1,\\
&\\
=|E_1|+|E_2|,& \mbox{ if } q=1,\\
&\\
<|E_1|^\frac{1}{q}+|E_2|^\frac{1}{q},& \mbox{ if } q>1.
\end{array}
\right.
\]
This in particular implies that for $q>1$, as observed in \cite{PS}, we have
\[
\begin{split}
\frac{\mathcal{H}^{N-1}(\partial (E_1\cup E_2))}{\left|E_1\cup E_2\right|^\frac{1}{q}}&>\frac{\mathcal{H}^{N-1}(\partial E_1)+\mathcal{H}^{N-1}(\partial E_2)}{|E_1|^\frac{1}{q}+|E_2|^\frac{1}{q}}\\
&\ge \min\left\{\frac{\mathcal{H}^{N-1}(\partial E_1)}{|E_1|^\frac{1}{q}},\frac{\mathcal{H}^{N-1}(\partial E_2)}{|E_2|^\frac{1}{q}}\right\}.
\end{split}
\]
Thus, in the minimization process for $q>1$, sets prefer not to break into multiple pieces. The previous argument obviously breaks down for $q<1$;
\vskip.2cm
\item for a {\it cylindrical set}, i.e. for a set of the form $E=\omega\times(0,L)$ for some $\omega\subseteq\mathbb{R}^{N-1}$ open and bounded, we have 
\[
\dfrac{\mathcal{H}^{N-1}(\partial E)}{|E|^\frac{1}{q}}\sim L^{1-\frac{1}{q}}\to \left\{\begin{array}{ll}
0, & \mbox{ if } 0<q<1,\\
1, & \mbox{ if } q=1,\\
+\infty,& \mbox{ if } q>1.
\end{array}
\right.
\]
Thus, in particular, if the open set $\Omega$ is unbounded in some direction and contains a ``tube'', minimizing sequences for $h_q(\Omega)$ may have the interest to ``stretch'' as much as possible in the case $q<1$, eventually preventing the existence of an optimal set. On the contrary, in light of the previous computation, for $q>1$ minimizing sequences should rather stay uniformly bounded, thus gaining some form of compactness.  
\end{enumerate}
These rough and clumsy observations may suggest that some new phenomena could appear in the minimization problem and that the two regimes $0<q<1$ and $q>1$ should give rise to qualitatively different ``results'', at least for {\it disconnected sets} and/or {\it very elongated sets}. Not surprisingly, we will see that this is exactly what happens.

\subsection{Main results}
After this lengthy presentation of the aims and scopes of the paper, we list below the main achievements of our discussion on the generalized Cheeger constant $h_q$. They are essentially of three different types:
\begin{enumerate}
\item {\bf Comparison of Cheeger's constants}: here we answer the question \eqref{Q2}. For every $1<q<N/(N-1)$ and every $\Omega\subseteq\mathbb{R}^N$ open set, we have
\begin{equation}
\label{comparison1}
C_1\,\Big(h_1(\Omega)\Big)^{\frac{N}{q}-(N-1)}\le h_q(\Omega)\le C_2\,\Big(h_1(\Omega)\Big)^{\frac{N}{q}-(N-1)},
\end{equation}
for two constants $C_1,C_2>0$ {\it depending on $N$ and $q$, only} (see Theorem \ref{thm:equivalenti}). On the other hand, for every $0<q<1$ and and every $\Omega\subseteq\mathbb{R}^N$  open set, we have 
\begin{equation}
\label{comparison2}
h_q(\Omega)\le C_3\,\Big(h_1(\Omega)\Big)^{\frac{N}{q}-(N-1)},
\end{equation}
for a constant $C_3>0$ depending on $N$ and $q$, only. Moreover, the estimate can not be reverted: indeed, we have
\[
h_1(\mathbb{R}^{N-1}\times(-1,1))>0\qquad \mbox{ and }\qquad h_q(\mathbb{R}^{N-1}\times(-1,1))=0,
\]
see Proposition \ref{prop:subbo}. 
\par
Observe that by combining \eqref{cheegerineq} and \eqref{comparison1} or \eqref{comparison2}, we can answer the question \eqref{Q1}, as well. Indeed, we obtain the following {\it generalized Cheeger inequality}
\begin{equation}
\label{generalcheeger}
\frac{1}{C}\,\Big(h_q(\Omega)\Big)^\frac{2}{\frac{N}{q}-(N-1)}\le \lambda(\Omega),
\end{equation}
with a constant $C>0$ depending on $N$ and $q$, only. This holds for every open set $\Omega\subseteq\mathbb{R}^N$ and every $0<q<N/(N-1)$.
\vskip.2cm
\item {\bf Poincar\'e constants}: for every $1<q<N/(N-1)$, the constant $h_q(\Omega)$ always coincides with a sharp Poincar\'e-Sobolev constant, i.e. 
\[
h_q(\Omega)=\lambda_{1,q}(\Omega):=\inf_{u\in C^\infty_0(\Omega)} \left\{\int_\Omega |\nabla u|\,dx\, :\, \int_\Omega |u|^q\,dx=1\right\}.
\]
On the other hand, for every $0<q<1$, in general we only have 
\[
h_q(\Omega)\ge \lambda_{1,q}(\Omega),
\]
see Lemma \ref{lemma:angeo}
and the equality breaks down for some particular shapes, see Example \ref{exa:counter}. This gives an answer to \eqref{Q3};
\vskip.2cm
\item {\bf Existence of extremals for convex sets}: in dimension $N=2$ and for every $1<q<2$, if $\Omega\subseteq\mathbb{R}^2$ is an open {\it convex} set (not necessarily bounded), then 
\[
h_q(\Omega)=\inf\left\{\frac{\mathcal{H}^{1}(\partial E)}{|E|^\frac{1}{q}}\, :\, E\subseteq \Omega \mbox{ open bounded {\it convex} set}\right\},
\]
a fact already observed in \cite{PS}. Here we prove that the infimum is attained {\it if and only if} the high ridge set of $\Omega$ is not empty, see Theorem \ref{thm:existence_convex}. We recall that the high ridge set is the following subset
\[
\mathcal{M}(\Omega):=\Big\{x\in\Omega\, :\, B_{r_\Omega}(x)\subseteq\Omega\Big\},
\] 
where $r_\Omega$ is the {\it inradius} of $\Omega$.
\par
On the other hand, for every $0<q<1$, for unbounded convex sets we always have $h_q(\Omega)=0$ and thus existence of a minimizer fails, see Lemma \ref{lemma:appu}. The borderline case $q=1$ is particular: here, it may happen that $h_1(\Omega)>0$ but the infimum is not attained, as for $\Omega=\mathbb{R}\times(-1,1)$ (see for example \cite[Theorm 3.1]{KP}).
\end{enumerate}
Some remarks are in order, for each of the previous results. We also list some open questions, that we think are quite interesting.
\begin{remark}
 We notice that the generalized Cheeger inequality \eqref{generalcheeger} improves that of \cite[Theorem 3.1]{Av}. While the latter depends on the volume term $|\Omega|$, the estimate \eqref{generalcheeger} holds for general open sets, even having infinite volume. As for estimates \eqref{comparison1} and \eqref{comparison2}, we point out that the constants $C_1$ and $C_3$ are explicit and sharp. On the contrary, the constant $C_2$ is explicit, but determining the value of the sharp constant is an intriguing open problem. 
\end{remark}

\begin{remark}
As we said above, the fact that $\lambda_{1,q}(\Omega)=h_q(\Omega)$ for $1<q<N/(N-1)$ was already known. On this point, the main new fact is the counter-example showing that in general we can not have the same result, for $0<q<1$. The counter-example crucially exploits the fact that for disconnected sets ``something strange'' could happen, as previously exposed. In light of this fact, one could (very) bravely guess that for convex sets, we could still have the equivalence between $\lambda_{1,q}$ and $h_q$, even for $0<q<1$. This is an aspect that deserves to be investigated in the future.
\end{remark}

\begin{remark}
We point out that our existence theorem for $1<q<2$ partially superposes with \cite[Theorem 4.2]{PS}. On the one hand, the latter is more general, since it deals with a class of open planar sets, not necessarily convex; on the other hand, it is more restrictive, since the class considered is that of ``strips'', i.e. neighborhoods of regular curves (possibly infinite). 
\par
Our existence proof is different from that of \cite{PS} and it is based on a ``four terms'' geometric inequality for convex sets, linking diameter, inradius, perimeter and area. This is taken from our previous paper \cite{Bra}. We point out that this argument, needed is order to infer compactness of minimizing sequences of convex sets, would work verbatim in every dimension $N\ge 2$. The restriction to the case of $N=2$ is needed in order to assure that we can restrict to convex subsets, without affecting the infimum $h_q(\Omega)$. As observed in \cite{PS}, in dimension $N=2$ this can be obtained by using that the convex hull of a connected set decreases the perimeter, while enlarging its area. In higher dimensions this fails to be true and thus more sophisticated arguments would be needed, maybe inspired to those used for the case $q=1$ (see for example \cite{ACC}). This is certainly an interesting point, that we leave for future research, as well.
\end{remark}

\subsection{Plan of the paper}
The paper consists of five sections, plus two appendices. In Section \ref{sec:2} we recall some basic facts about $h_q$ and the Poincar\'e-Sobolev constants $\lambda_{1,q}$. Section \ref{sec:3} is devoted to briefly discuss the case of disconnected sets. We discuss the equivalence of generalized Cheeger's constants in Section \ref{sec:4}. Finally, in Section \ref{sec:5} we present the case of convex planar sets. Appendix \ref{sec:A} contains a classical local $L^\infty-L^q$ estimate for the so-called {\it $p-$torsion function}, while Appendix \ref{sec:B} records a simple approximation results for convex sets.

\begin{ack}
It is a pleasure to acknowledge some conversations with Giorgio Saracco, who also pointed out the references \cite{ACC} and \cite{Av}. We thank Nicola Fusco and Paolo Salani, for their insights on a couple of points of this paper.
Remark \ref{rem:simon} comes from a conversation with Simon Larson at the Institute Mittag-Leffler, in September 2022: we wish to thank him. Finally, some of the contents of this paper have been presented during a seminar in Bielefeld in January 2024: we wish to thank Anna Balci and Lars Diening for their kind invitation and the friendly atmosphere provided during the staying.
\end{ack}

\section{Preliminaries}
\label{sec:2}

We indicate by $B_R(x_0)$ the $N-$dimensional open ball centered at $x_0\in\mathbb{R}^N$, having radius $R>0$. For balls centered at the origin, we will simply write $B_R$. By $\omega_N$ we mean the volume of $B_1$. Occasionally, we will also need to work with cubes: we will use the symbol
\[
Q_R(x_0)=\prod_{i=1}^N (x_0^i-R,x_0^i+R),\qquad \mbox{ with } x_0=(x_0^1,\dots,x_0^N).
\]
Here as well, we will simply write $Q_R$ when the center $x_0$ is the origin.
\par
For an open set $\Omega\subseteq\mathbb{R}^N$ and a pair of exponents $1<p<\infty$, $0<q<\infty$, we will use the following notation
\[
\lambda_{p,q}(\Omega)=\inf_{u\in C^\infty_0(\Omega)} \left\{\int_\Omega |\nabla u|^p\,dx\, :\, \int_\Omega |u|^q\,dx=1\right\}.
\]
As exposed above, in this paper we want to consider the following geometric constant
\[
h_q(\Omega)=\inf\left\{\frac{\mathcal{H}^{N-1}(\partial E)}{|E|^\frac{1}{q}}\, :\, E\Subset\Omega \mbox{ open set with smooth boundary}\right\},
\]
for every $0<q<N/(N-1)$.
We notice that if $\Omega=B_R(x_0)$, by the Isoperimetric Inequality we get that
\begin{equation}
\label{palle}
h_q(B_R(x_0))=\frac{\mathcal{H}^{N-1}(\partial B_R(x_0))}{|B_R(x_0)|^\frac{1}{q}}=N\,\omega_N^{1-\frac{1}{q}}\,R^{N-1-\frac{N}{q}},
\end{equation}
as already observed in \cite[Section 2]{PS}. We now briefly explain the restrictions on $q$ (see also \cite[Remark 2.1]{PS}).
\begin{remark}[Limit cases]
\label{rem:restrizioni}
For $N\ge 2$, in the case $q=N/(N-1)$ the constant $h_q(\Omega)$ is not interesting. Indeed, it does not depend on $\Omega$ and it simply coincides with
\[
h_\frac{N}{N-1}(\Omega)=N\,\omega_N^\frac{1}{N},
\]
i.e. this is the sharp Euclidean isoperimetric constant. 
\par
Analogously, for $q>N/(N-1)$ the constant $h_q(\Omega)$ is not interesting, as well: in this case, taken $B_{r_0}(x_0)\subseteq \Omega$, we have $B_{r}(x_0)\Subset \Omega$ for every $0<r<r_0$ and thus by \eqref{palle}
\[
h_q(\Omega)\le \lim_{r\searrow 0}\frac{\mathcal{H}^{N-1}(\partial B_r(x_0))}{|B_r(x_0)|^\frac{1}{q}}=N\,\omega_N^{1-\frac{1}{q}}\,\lim_{r\searrow 0}r^{N-1-\frac{N}{q}}=0.
\]
Finally, in the case $N=1$, we could allow $q$ to take the limit value $q=\infty$, with the understanding that $1/q=0$ and $|E|^0=1$. In this case, for every non-empty open set $\Omega\subseteq\mathbb{R}$ we have 
\[
h_\infty(\Omega)=\inf\left\{\mathcal{H}^0(\partial E)\, :\, E\Subset\Omega \mbox{ open set with smooth boundary}\right\}=2.
\]
The infimum is attained by any interval $(a,b)\subseteq\Omega$.
\end{remark}
We start with a standard result, showing the relation between Cheeger constants and principal frequencies. As recalled in in the Introduction, this is well-known: however, we reproduce the proof for the reader's convenience. This will also permit us to highlight a first difference between the two regimes $q<1$ and $q>1$.
\begin{lemma}
\label{lemma:angeo}
Let $\Omega\subseteq\mathbb{R}^N$ be an open set and $0< q<N(N-1)$. Then we have
\[
h_q(\Omega)\ge \lambda_{1,q}(\Omega).
\]
Moreover, if $1\le q<N/(N-1)$ the two quantities coincide.
\end{lemma}
\begin{proof}
For every $E\Subset \Omega$ with smooth boundary, there exists a sequence $\{\varphi_n\}_{n\in\mathbb{N}}\subseteq C^\infty_0(\Omega)$ such that
\begin{equation}
\label{approx}
\lim_{n\to\infty} \|\varphi_n-1_E\|_{L^q(\Omega)}=0 \qquad \mbox{ and }\qquad \lim_{n\to\infty} \|\nabla \varphi_n\|_{L^1(\Omega)}=|\nabla 1_E|(\mathbb{R}^N)=\mathcal{H}^{N-1}(\partial E).
\end{equation}
Here $1_E$ is the characteristic function of $E$.
Such a sequence can be constructed by means of standard convolution methods: we set 
\[
\varphi_n=1_E\ast \varrho_n,
\]
where 
$\{\varrho_n\}_{n\ge 1}$ is the usual family of standard mollifiers.
By observing that $\varrho_n\in C^\infty_0(B_{2/n})$ and that $E\Subset \Omega$, we get
\[
\varphi_n\in C^\infty_0(\Omega),\qquad \mbox{ for $n$ large enough}.
\]
The first property in \eqref{approx} easily follows from the properties of convolutions, while by \cite[page 121]{AFP} we have 
\[
\lim_{n\to\infty}\|\nabla \varphi_n\|_{L^1(\Omega)}=\lim_{n\to\infty} |\nabla \varphi_n|(\mathbb{R}^N)=|\nabla 1_E|(\mathbb{R}^N),
\]
and the last term coincides with $\mathcal{H}^{N-1}(\partial E)$ by \cite[Proposition 3.62]{AFP}.
We thus get from \eqref{approx}
\[
\frac{\mathcal{H}^{N-1}(\partial E)}{|E|^\frac{1}{q}}=\lim_{n\to\infty} \frac{\displaystyle\int_\Omega |\nabla \varphi_n|\,dx}{\displaystyle\left(\int_\Omega |\varphi_n|^q\,dx\right)^\frac{1}{q}}\ge \lambda_{1,q}(\Omega).
\]
By arbitrariness of $E$, we get the inequality $h_q(\Omega)\ge \lambda_{1,q}(\Omega)$.
\vskip.2cm\noindent
We now assume that $1\le q<N/N(N-1)$ and prove the converse inequality. Let $\varphi\in C^\infty_0(\Omega)\setminus\{0\}$, by using the Coarea Formula, Sard's Theorem and the definition of $h_q(\Omega)$, we get
\[
\int_\Omega |\nabla \varphi|\,dx=\int_0^{+\infty} \mathcal{H}^{N-1}\left(\Big\{x\in\Omega\,:\,|\varphi(x)|=t\Big\}\right)\,dt\ge h_q(\Omega)\,\int_0^{+\infty} \left|\Big\{x\in\Omega\,:\,|\varphi(x)|>t\Big\}\right|^\frac{1}{q}\,dt.
\]
By using Cavalieri's principle and \cite[Section 1.3.5, Lemma 1]{Maz}, we get\footnote{This passage crucially exploits the fact that $q\ge 1$. For $0<q<1$ the inequality is reverted. This inequality can be rephrased by saying that the we have the following continuous embedding $L^{q,1}(\Omega)\hookrightarrow L^{q,q}(\Omega)=L^q(\Omega)$ between {\it Lorentz spaces}. This holds only for $q\ge 1$. For $0<q<1$, the situation is reverted and the Lorentz space $L^{q,1}(\Omega)$ is actually larger than $L^q(\Omega)$.}
\[
\left(\int_\Omega |\varphi|^q\,dx\right)^\frac{1}{q}=\left(q\,\int_0^{+\infty} t^{q-1}\,\left|\Big\{x\in\Omega\,:\,|\varphi(x)|>t\Big\}\right|\,dt\right)^\frac{1}{q}\le \int_0^{+\infty} \left|\Big\{x\in\Omega\,:\,|\varphi(x)|>t\Big\}\right|^\frac{1}{q}\,dt.
\]
The two estimates above prove that 
\[
\frac{\displaystyle\int_\Omega |\nabla \varphi|\,dx}{\displaystyle\left(\int_\Omega |\varphi|^q\,dx\right)^\frac{1}{q}}\ge h_q(\Omega).
\]
By arbitrariness of $\varphi$, this gives that $\lambda_{1,q}(\Omega)\ge h_q(\Omega)$, as well.
\end{proof}
\begin{remark}[The case $0<q<1$]
We observe that for $0<q<1$ the previous proof does not permit to infer that
\[
\lambda_{1,q}(\Omega)\ge h_q(\Omega).
\]
There is a good reason for this fact: actually, in general the two quantities {\it do not coincide}. We refer to Example \ref{exa:counter} for a counter-example.
\end{remark}
The next result is quite classical, as well. However, we try to keep the assumptions on the open sets at a minimal level.
\begin{lemma}
\label{lemma:limitescemo}
Let $\Omega\subseteq\mathbb{R}^N$ be an open set. For every $0< q<N/(N-1)$, we have 
\[
\limsup_{p\searrow 1} \lambda_{p,q}(\Omega)\le \lambda_{1,q}(\Omega).
\]
Moreover, if $\Omega$ has finite volume, then we also have
\begin{equation}
\label{liminf}
\liminf_{p\searrow 1} \lambda_{p,q}(\Omega)\ge \lambda_{1,q}(\Omega).
\end{equation}
\end{lemma}
\begin{proof}
For every $\varphi\in C^\infty_0(\Omega)$ not identically vanishing, we have 
\[
\limsup_{p\searrow 1} \lambda_{p,q}(\Omega)\le \lim_{p\searrow 1}\frac{\displaystyle\int_\Omega |\nabla \varphi|^p\,dx}{\displaystyle \left(\int_\Omega |\varphi|^q\,dx\right)^\frac{p}{q}}=\frac{\displaystyle\int_\Omega |\nabla \varphi|\,dx}{\displaystyle \left(\int_\Omega |\varphi|^q\,dx\right)^\frac{1}{q}}.
\]
By taking the infimum over $C^\infty_0(\Omega)$, we thus get 
\[
\limsup_{p\searrow 1} \lambda_{p,q}(\Omega)\le \lambda_{1,q}(\Omega).
\]
We now assume that $|\Omega|<+\infty$. To prove the $\liminf$ inequality, we take $\varphi\in C^\infty_0(\Omega)\setminus\{0\}$ and then observe that by H\"older's inequality
\[
\lambda_{1,q}(\Omega)\le \frac{\displaystyle\int_\Omega |\nabla \varphi|\,dx}{\displaystyle \left(\int_\Omega |\varphi|^{q}\,dx\right)^\frac{1}{q}}\le |\Omega|^{1-\frac{1}{p}}\,\frac{\displaystyle\left(\int_\Omega |\nabla \varphi|^p\,dx\right)^\frac{1}{p}}{\displaystyle \left(\int_\Omega |\varphi|^{q}\,dx\right)^\frac{1}{q}}. 
\]
This shows that 
\[
\lambda_{1,q}(\Omega)\le |\Omega|^{1-\frac{1}{p}}\,\Big(\lambda_{p,q}(\Omega)\Big)^\frac{1}{p}.
\]
By taking the $\liminf$ as $p$ goes to $1$, we conclude.
\end{proof}
\begin{remark}
The assumption $|\Omega|<+\infty$ is probably not optimal for \eqref{liminf} to hold, but we may notice that for general open sets this result can not hold. For example, for $q=1$ and 
\[
\Omega=\mathbb{R}^{N-1}\times(-1,1),
\]
we have
\[
\lambda_{p,1}(\Omega)=0,\ \mbox{ for every } 1<p<\infty,\qquad \mbox{ while }\qquad \lambda_{1,1}(\Omega)>0.
\]
Thus, in this case \eqref{liminf} can not hold. We recall that the equality $\lambda_{p,1}(\Omega)=0$ follows from the fact that\footnote{We denote by $\mathcal{D}^{1,p}_0(\Omega)$ the homogeneous Sobolev space obtained as the completion of $C^\infty_0(\Omega)$ with respect to the norm
\[
\varphi\mapsto \|\nabla \varphi\|_{L^p(\Omega)}.
\]
We also recall that the property $\lambda_{p,1}(\Omega)>0$ entails that $\mathcal{D}^{1,p}_0(\Omega)$ coincides with the more familiar $W^{1,p}_0(\Omega)$ (see for example \cite[Proposition 2.4]{BPZ}), the latter being the closure of $C^\infty_0(\Omega)$ in the usual Sobolev space $W^{1,p}(\Omega)$.} 
\[
\lambda_{p,1}(\Omega)>0 \qquad \Longleftrightarrow \qquad \mathcal{D}^{1,p}_0(\Omega)\hookrightarrow L^1(\Omega) \mbox{ is compact},
\]
see \cite[Theorem 15.6.2]{Maz} and also \cite[Theorem 1.2]{BF}. For the set $\Omega=\mathbb{R}^{N-1}\times (-1,1)$, the invariance by translation in the first $N-1$ directions makes such an embedding non-compact. 
\end{remark}

\section{Disconnected sets}
\label{sec:3}

We begin with a very simple result for a particular class of real functions of one variable. The proof is omitted, it is just based on very standard facts.
\begin{lemma}
\label{lemma:ulanbator}
Let $a,b\ge 0$ and $c,d>0$, for $\beta>0$ we define the function
\[
\phi_\beta(t)=\frac{a+t\,b}{(c+t^\beta\,d)^\frac{1}{\beta}},\qquad \mbox{ for every } t>0.
\]
For $\beta\ge 1$, we have 
\[
\phi_\beta(t)>\min\Big\{\phi_\beta(0),\lim_{t\to+\infty} \phi_\beta(t)\Big\}= \min\left\{\frac{a}{c^\frac{1}{\beta}},\frac{b}{d^\frac{1}{\beta}}\right\},\qquad \mbox{ for every } t>0.
\]
For $0<\beta<1$, we have 
\[
\phi_\beta(t)\ge \displaystyle\left(\left(\frac{c^\frac{1}{\beta}}{a}\right)^\frac{\beta}{1-\beta}+\left(\frac{d^\frac{1}{\beta}}{b}\right)^\frac{\beta}{1-\beta}\right)^\frac{\beta-1}{\beta},\qquad \mbox{ for every } t>0,
\]
with equality if and only if
\[
t=t_\beta:=\left(\frac{a}{c}\,\frac{d}{b}\right)^\frac{1}{1-\beta}.
\]
\end{lemma}
The previous result permits us to compute $\lambda_{1,q}$ for disconnected sets. 
\begin{lemma}[Case $q\ge 1$]
\label{lemma:spinformula}
Let $\Omega_1,\Omega_2\subseteq \mathbb{R}^N$ be two open sets, such that $\Omega_1\cap \Omega_2=\emptyset$. Then for $1\le q<N/(N-1)$, we have 
\[
h_q(\Omega_1\cup \Omega_2)=\lambda_{1,q}(\Omega_1\cup \Omega_2)=\min\Big\{\lambda_{1,q}(\Omega_1),\, \lambda_{1,q}(\Omega_2)\Big\}=\min\Big\{h_q(\Omega_1),\, h_q(\Omega_2)\Big\} .
\]
\end{lemma}
\begin{proof}
It is sufficient to prove the second equality for $\lambda_{1,q}(\Omega_1\cup \Omega_2)$, the others being a consequence of Lemma \ref{lemma:angeo}.
We take $\varphi_1\in C^\infty_0(\Omega_1)\setminus\{0\}$ and $\varphi_2\in C^\infty_0(\Omega_2)\setminus\{0\}$. Then we have $\varphi_1+t\,\varphi_2\in C^\infty_0(\Omega_1\cup\Omega_2)$, for every $t>0$. Accordingly, we get
\[
\lambda_{1,q}(\Omega_1\cup \Omega_2)\le \frac{\displaystyle\int_{\Omega_1} |\nabla \varphi_1|\,dx+t\,\int_{\Omega_2} |\nabla \varphi_2|\,dx}{\displaystyle\left(\int_{\Omega_1} |\varphi_1|^q\,dx+t^q\,\int_{\Omega_2} |\varphi_2|^q\,dx\right)^\frac{1}{q}}.
\]
We can now minimize with respect to $t>0$: thanks to Lemma \ref{lemma:ulanbator}, for $1\le q<N/(N-1)$ we get
\[
\lambda_{1,q}(\Omega)\le \min\left\{\frac{\displaystyle\int_{\Omega_1} |\nabla \varphi_1|\,dx}{\displaystyle\left(\int_{\Omega_1} |\varphi_1|^q\,dx\right)^\frac{1}{q}},\, \frac{\displaystyle \int_{\Omega_2} |\nabla \varphi_2|\,dx}{\displaystyle\left(\int_{\Omega_2} |\varphi_2|^q\,dx\right)^\frac{1}{q}} \right\}.
\]
By arbitrariness of $\varphi_1$ and $\varphi_2$, we thus get 
\[
\lambda_{1,q}(\Omega_1\cup \Omega_2)\le \min\Big\{\lambda_{1,q}(\Omega_1),\, \lambda_{1,q}(\Omega_2)\Big\}.
\]
In order to prove the reverse inequality, we take $\varphi\in C^\infty_0(\Omega_1\cup \Omega_2)\setminus\{0\}$. 
We call $\varphi_1$ and $\varphi_2$ the restrictions of $\varphi$ to $\Omega_1$ and $\Omega_2$, respectively. Of course, we have $\varphi_i\in C^\infty_0(\Omega_i)$, for $i=1,2$. Let us suppose at first that $\varphi_1\not\equiv 0$ and $\varphi_2\not\equiv 0$. Then, if we set 
\[
\phi_q(t):=\frac{\displaystyle\int_{\Omega_1} |\nabla \varphi_1|\,dx+t\,\int_{\Omega_2} |\nabla \varphi_2|\,dx}{\displaystyle\left(\int_{\Omega_1} |\varphi_1|^q\,dx+t^q\,\int_{\Omega_2} |\varphi_2|^q\,dx\right)^\frac{1}{q}},\qquad \mbox{ for } t>0,
\]
again by Lemma \ref{lemma:ulanbator} we have 
\[
\begin{split}
\frac{\displaystyle\int_{\Omega_1} |\nabla \varphi_1|\,dx+\int_{\Omega_2} |\nabla \varphi_2|\,dx}{\displaystyle\left(\int_{\Omega_1} |\varphi_1|^q\,dx+\int_{\Omega_2} |\varphi_2|^q\,dx\right)^\frac{1}{q}}=\phi_q(1)\ge \inf_{t>0}\phi_q(t)&=\min\left\{\frac{\displaystyle\int_{\Omega_1} |\nabla \varphi_1|\,dx}{\displaystyle\left(\int_{\Omega_1} |\varphi_1|^q\,dx\right)^\frac{1}{q}},\, \frac{\displaystyle \int_{\Omega_2} |\nabla \varphi_2|\,dx}{\displaystyle\left(\int_{\Omega_2} |\varphi_2|^q\,dx\right)^\frac{1}{q}} \right\}\\
&\ge \min\Big\{\lambda_{1,q}(\Omega_1), \lambda_{1,q}(\Omega_2)\Big\}.
\end{split}
\]
On the other hand, it easily seen that the same lower bound holds if $\varphi_1\equiv 0$ or $\varphi_2\equiv 0$, as well. By arbitrariness of $\varphi\in C^\infty_0(\Omega_1\cup \Omega_2)$, we thus get the claimed equality for $\lambda_{1,q}(\Omega)$. 
\end{proof}
In a similar way, we can cover the case $0<q<1$, as well. In particular, we extend to the case $p=1$ and $0<q<1$ a formula obtained in \cite[Corollary 2.4]{BF}. As simple as it is, it will be useful in order to construct some counter-examples.
\begin{lemma}[Case $0<q<1$]
\label{lemma:spinformulab}
Let $\Omega_1,\Omega_2\subseteq \mathbb{R}^N$ be two open sets, such that $\Omega_1\cap \Omega_2=\emptyset$. Then for $0<q<1$: 
\begin{itemize}
\item if $\lambda_{1,q}(\Omega_1)>0 $ and $\lambda_{1,q}(\Omega_2)>0$ we have 
\begin{equation}
\label{spin!spin!}
\lambda_{1,q}(\Omega_1\cup \Omega_2)=\displaystyle\left(\left(\frac{1}{\lambda_{1,q}(\Omega_1)}\right)^\frac{q}{1-q}+\left(\frac{1}{\lambda_{1,q}(\Omega_2)}\right)^\frac{q}{1-q}\right)^\frac{q-1}{q};
\end{equation}
\item if $\lambda_{1,q}(\Omega_1)=0$ or $\lambda_{1,q}(\Omega_2)=0$, then we have 
\[
\lambda_{1,q}(\Omega_1\cup \Omega_2)=0.
\]
\end{itemize}
\end{lemma}
\begin{proof}
We assume at first that both $\lambda_{1,q}(\Omega_1)$ and $\lambda_{1,q}(\Omega_2)$ are positive. The upper bound
\[
\lambda_{1,q}(\Omega_1\cup \Omega_2)\le \displaystyle\left(\left(\frac{1}{\lambda_{1,q}(\Omega_1)}\right)^\frac{q}{1-q}+\left(\frac{1}{\lambda_{1,q}(\Omega_2)}\right)^\frac{q}{1-q}\right)^\frac{q-1}{q},
\]
can be obtained as in the first part of proof of Lemma \ref{lemma:spinformula}, by using this time Lemma \ref{lemma:ulanbator} for $0<\beta=q<1$. 
\par
In order to prove the reverse inequality, it is now useful to observe that the right-hand side is strictly smaller than both values $\lambda_{1,q}(\Omega_1)$ and $\lambda_{1,q}(\Omega_2)$. This is due to the fact that the function of two real variables
\[
(t,s)\mapsto \left(t^\frac{q}{1-q}+s^\frac{q}{1-q}\right)^\frac{q-1}{q},\qquad \mbox{ for } (t,s)\in (0,+\infty)\times(0,+\infty),
\]
is {\it decreasing} in both variables, thanks to the fact that $q-1<0$. In particular, we have 
\[
\left(t^\frac{q}{1-q}+s^\frac{q}{1-q}\right)^\frac{q-1}{q}<\frac{1}{t}\qquad \mbox{ and }\qquad \left(t^\frac{q}{1-q}+s^\frac{q}{1-q}\right)^\frac{q-1}{q}<\frac{1}{s},
\]
for every $t,s>0$. This very simple observation shows that $\lambda_{1,q}(\Omega_1\cup \Omega_2)$ can be equivalently defined, by restricting the minimization to functions $\varphi\in C^\infty_0(\Omega_1\cup \Omega_2)\setminus\{0\}$ such that both the restriction of $\varphi$ to $\Omega_1$ and that to $\Omega_2$ are not identically vanishing. In light of this fact, we can now run the same argument as in the second part of the proof of Lemma \ref{lemma:spinformula}.
\par
We now suppose that $\lambda_{1,q}(\Omega_2)=0$, for example. We set $\Omega_{i,n}=\Omega_i\cap B_n$ for $i=1,2$, which are not empty for $n$ large enough. By assumption and using that $\{\Omega_{i,n}\}_{n\in\mathbb{N}}$ is an exhaustion of $\Omega_i$, we get
\[
\lim_{n\to\infty} \lambda_{1,q}(\Omega_{i,n})=\lambda_{1,q}(\Omega_i),\qquad \mbox{ for } i=1,2,
\]
and
\[
\lim_{n\to\infty} \lambda_{1,q}(\Omega_{1,n}\cup \Omega_{2,n})=\lambda_{1,q}(\Omega_1\cup \Omega_2).
\]
By observing that $\Omega_{i,n}$ is an open bounded set, we have that $\lambda_{1,q}(\Omega_{i,n})>0$, for every $n$ large enough. We can thus apply formula \eqref{spin!spin!} to $\Omega_{1,n}\cup \Omega_{2,n}$ and pass to the limit, in order to get
\[
\lambda_{1,q}(\Omega_1\cup \Omega_2)=\lim_{n\to\infty} \lambda_{1,q}(\Omega_{1,n}\cup \Omega_{2,n})=\lim_{n\to\infty}\displaystyle\left(\left(\frac{1}{\lambda_{1,q}(\Omega_{1,n})}\right)^\frac{q}{1-q}+\left(\frac{1}{\lambda_{1,q}(\Omega_{2,n})}\right)^\frac{q}{1-q}\right)^\frac{q-1}{q}.
\]
By using that $q-1<0$ and that $\lambda_{1,q}(\Omega_{2,n})$ converges to $0$, we conclude. 
\end{proof}
\begin{remark}
The previous results can be easily iterated. Thus, for an open set of the form 
\[
\Omega=\bigcup_{n\in\mathbb{N}} \Omega_n,\qquad \mbox{  with } \Omega_i\cap \Omega_j=\emptyset \mbox{ for } i\not=j,
\]
we get for $1\le q<N/(N-1)$
\[
\lambda_{1,q}(\Omega)=h_q(\Omega)=\inf\Big\{h_q(\Omega_n)\, :\, n\in\mathbb{N}\Big\} =\inf\Big\{\lambda_{1,q}(\Omega_n)\, :\, n\in\mathbb{N}\Big\},
\]
and for $0<q<1$ 
\[
\lambda_{1,q}(\Omega)=\displaystyle\left(\sum_{n\in \mathbb{N}}\left(\frac{1}{\lambda_{1,q}(\Omega_n)}\right)^\frac{q}{1-q}\right)^\frac{q-1}{q},
\]
provided that $\lambda_{1,q}(\Omega_n)>0$, for every $n\in\mathbb{N}$.
\end{remark}
Thanks to the formula of Lemma \ref{lemma:spinformulab}, we can highlight some weird phenomena of $\lambda_{1,q}$ and $h_q$, for $0<q<1$. In particular, we show that in general {\it it is no more true that} $\lambda_{1,q}=h_q$.
\begin{example}
\label{exa:counter}
In dimension $N= 2$, we take the open set 
\[
\Omega=B_r(x_0)\cup B_R(y_0),\qquad \mbox{ with } 0<r,R\, \mbox{ and }\, |x_0-y_0|>r+R.
\] 
We take the following restriction
\begin{equation}
\label{raggioexa}
\frac{r}{R}<\frac{\sqrt{7}-2}{3}.
\end{equation}
For $q=1/2$, we are going to show that for this set we have
\[
\lambda_{1,\frac{1}{2}}(\Omega)<h_\frac{1}{2}(\Omega).
\]
By definition, we have
\[
h_\frac{1}{2}(\Omega)=\inf\left\{\frac{\mathcal{H}^{1}(\partial E_1)+\mathcal{H}^{1}(\partial E_2)}{\left(|E_1|+|E_2|\right)^2}\, :\, \begin{array}{c}E_1\Subset B_r(x_0) \mbox{ open set with smooth boundary}\\
E_2\Subset B_R(y_0) \mbox{ open set with smooth boundary}\\
E_1\cup E_2\not=\emptyset
\end{array}, 
\right\}.
\]
By using the Isoperimetric Inequality, we immediately see that this is the same as
\[
\begin{split}
h_\frac{1}{2}(\Omega)&=\inf\left\{\frac{\mathcal{H}^{1}(\partial B_{r_1}(x_1))+\mathcal{H}^{1}(\partial B_{r_2}(y_1))}{\left(|B_{r_1}(x_1)|+|B_{r_2}(y_1)|\right)^2}\, :\, \begin{array}{c}B_{r_1}(x_1)\Subset B_r(x_0) \mbox{ ball}\\
B_{r_2}(y_1)\Subset B_R(y_0) \mbox{ ball}
\end{array}, r_1+r_2>0
\right\}\\
&=\frac{2}{\pi}\,\inf\left\{\frac{r_1+r_2}{\left(r_1^2+r_2^2\right)^2}\, :\, r_1+r_2>0,\ r_1<r \mbox{ and } r_2<R\right\}.
\end{split}
\]
In order to determine this infimum, we need to find the infimum of the function
\[
\psi(t,s)=\frac{t+s}{\left(t^2+s^2\right)^2},\qquad \mbox{ for } (t,s)\in\Big([0,r)\times[0,R)\Big)\setminus\{(0,0)\}.
\]
We first observe that $\psi$ has no {\it internal} critical points: indeed, we have
\[
\nabla \psi(t,s)=(0,0)\qquad \Longleftrightarrow \qquad \left\{\begin{array}{ccc}
t^2+s^2&=&4\,(t+s)\,t\\
t^2+s^2&=&4\,(t+s)\,s
\end{array}
\right.
\] 
Since $t+s>0$, the last two equations imply that each critical point should have $t=s$. By spending this information in one of the two equations, we would get
\[
2\,t^2=8\,t^2.
\]
Thus, the only critical point would be $(0,0)$, which however lays on the boundary. We now study the restriction of $\psi$ to the boundary of its domain of definition. We first observe that
\[
\psi(t,0)=\frac{1}{t^3}\ge \frac{1}{r^3}=\psi(r,0),\qquad \mbox{ for } 0<t\le r,
\]
and
\[
\psi(0,s)=\frac{1}{s^3}\ge \frac{1}{R^3}=\psi(0,R),\qquad \mbox{ for } 0<s\le R.
\]
Since by \eqref{raggioexa} we have $r<R$, points of the form $(t,0)$ can not be minimizers of $\psi$. We now study the restriction $\psi(t,R)$: we have
\[
\begin{split}
\frac{d}{dt} \psi(t,R)\ge 0&\qquad \Longleftrightarrow \qquad t^2+R^2\ge 4\,(t+R)\,t\\
&\qquad \Longleftrightarrow \qquad 3\,t^2+4\,R\,t-R^2\le 0\\
&\qquad \Longleftrightarrow \qquad -\frac{2+\sqrt{7}}{3}\, R\le t\le \frac{-2+\sqrt{7}}{3}\, R.
\end{split}
\]
By recalling that $t\in(0,r]$ and the restriction \eqref{raggioexa}, we obtain that the last condition is always satisfied. Thus, the function $t\mapsto \psi(t,R)$ is increasing, which implies that
\[
\psi(0,R)<\psi(t,R)< \psi(r,R),\qquad \mbox{ for } t\in(0,r).
\]
Thus, points of the form $(t,R)$ can not be minimizers of $\psi$. Finally, we need to study the restriction $\psi(r,s)$: by symmetry, we have
\[
\begin{split}
\frac{d}{ds} \psi(r,s)\ge 0 &\qquad \Longleftrightarrow \qquad -\frac{2+\sqrt{7}}{3}\, r\le s\le \frac{-2+\sqrt{7}}{3}\, r.
\end{split}
\]
Thus, on $(0,R)$ we have two intervals of monotonicity: first $\psi(r,s)$ increases, and then it decreases. In particular, we get
\[
\psi(r,s)>\min\{\psi(r,0),\psi(r,R)\},\qquad \mbox{ for } s\in(0,R).
\]
By the previous discussion, the last two values have already been shown not to correspond to the minimum of $\psi$. At last, we get
\[
\psi(t,s)\ge \psi(0,R)=\frac{1}{R^3},
\qquad \mbox{ for } (t,s)\in\Big([0,r)\times[0,R)\Big)\setminus\{(0,0)\},
\]
that is the function $\psi$ is minimal for $t=0$ and $s=R$.
By going back to our Cheeger constant $h_{1/2}(\Omega)$, we found that
\[
h_\frac{1}{2}(\Omega)=\frac{2}{\pi\,R^3}=h_\frac{1}{2}(B_R(y_0)),
\]
where the last identity follows from \eqref{palle} with $N=2$ and $q=1/2$.
By relying on this identity, we can finally prove that $\lambda_{1,1/2}(\Omega)<h_{1/2}(\Omega)$. Indeed, by Lemma \ref{lemma:spinformulab} with $q=1/2$, we have
\[
\begin{split}
\lambda_{1,\frac{1}{2}}(\Omega)&=\displaystyle\left(\frac{1}{\lambda_{1,\frac{1}{2}}(B_{r}(x_0))}+\frac{1}{\lambda_{1,\frac{1}{2}}(B_{R}(y_0))}\right)^{-1}<\lambda_{1,\frac{1}{2}}(B_R(y_0)).
\end{split}
\]
Finally, we can estimate the last term by Lemma \ref{lemma:angeo}, this entails that
\[
\lambda_{1,\frac{1}{2}}(\Omega)<\lambda_{1,\frac{1}{2}}(B_R(y_0))\le h_\frac{1}{2}(B_R(y_0))=h_\frac{1}{2}(\Omega).
\]
This gives the desired conclusion.
\end{example}

\begin{remark}
Though in general $\lambda_{1,q}(\Omega)$ does not coincide with $h_q(\Omega)$ for $0<q<1$, we recall that it is possible to give a characterization of sets for which $\lambda_{1,q}(\Omega)$ is positive, in terms of isoperimetric--like constants. This is due to Maz'ya, we refer to \cite[Theorem 2.1.4]{Maz} for more details.
\end{remark}

\section{Comparison of generalized Cheeger's constants}
\label{sec:4}

The next result shows that for an open set, {\it all the generalized Cheeger constants are actually equivalent}, provided $q>1$. In light of Lemma \ref{lemma:angeo}, this result is implicitly contained in \cite[Theorem 15.4.1]{Maz}. We give however a different proof, based on PDE methods. This also produces an explicit constant, which is very likely not optimal.
\begin{theorem}[The case $q>1$]
\label{thm:equivalenti}
Let $N\ge 2$. For every $1<q<N/(N-1)$, there exists an explicit constant $C=C(N,q)>0$ such that for every open set $\Omega\subseteq\mathbb{R}^N$ we have
\[
\left(N\,\omega_N^\frac{1}{N}\right)^{N-\frac{N}{q}}\,\Big(h_1(\Omega)\Big)^{\frac{N}{q}-(N-1)}\le h_q(\Omega)\le C\,\Big(h_1(\Omega)\Big)^{\frac{N}{q}-(N-1)}.
\]
The leftmost inequality is sharp, equality being attained for $N-$dimensional balls.
\end{theorem}
\begin{proof}
The first inequality is quite easy, it is a straightforward consequence of the Isoperimetric Inequality. Indeed, for every $E\Subset\Omega$ with smooth boundary, we can write
\begin{equation}
\label{decompose}
\frac{\mathcal{H}^{N-1}(\partial E)}{|E|^\frac{1}{q}}=\left(\frac{\mathcal{H}^{N-1}(\partial E)}{|E|}\right)^{\frac{N}{q}-(N-1)}\,\left(\frac{\mathcal{H}^{N-1}(\partial E)}{|E|^\frac{N-1}{N}}\right)^{N-\frac{N}{q}}.
\end{equation}
By using the Isoperimetric Inequality for the second term and the fact that $N-N/q>0$ for $q>1$, we have
\[
\left(\frac{\mathcal{H}^{N-1}(\partial E)}{|E|^\frac{N-1}{N}}\right)^{N-\frac{N}{q}}\ge \Big(N\,\omega_N^\frac{1}{N}\Big)^{N-\frac{N}{q}}.
\]
This immediately gives the leftmost inequality, together with the equality cases, in light of \eqref{palle}.
\vskip.2cm\noindent
The converse inequality is more delicate. We will adapt an idea taken from \cite[Theorem 9]{vanBu}, which involves an $L^\infty$ bound for the so-called {\it $p-$torsion function} of a set.
\par
We first suppose that $\Omega\subseteq\mathbb{R}^N$ is bounded, with smooth boundary. For every $1<p<2$, we take $w_{\Omega,p}$ to be the $p-$torsion function of $\Omega$, i.e. $w_{p,\Omega}\in W^{1,p}_0(\Omega)$ is the unique weak solution belonging to $W^{1,p}_0(\Omega)$ of the equation
\[
-\Delta_p u=1,\qquad \mbox{ in }\Omega.
\]
In other words, we have 
\[
\int_\Omega \langle |\nabla w_{\Omega,p}|^{p-2}\,\nabla w_{\Omega,p},\nabla \varphi\rangle\,dx=\int_\Omega \varphi\,dx,\qquad \mbox{ for every } \varphi\in W^{1,p}_0(\Omega).
\]
We observe that $w_{\Omega,p}\in L^\infty(\Omega)$, by standard regularity results (see for example \cite[Proposition 3.1]{BR} or \cite[Proposition 6]{BE} for an explicit estimate). Moreover, since we are assuming that $\Omega$ is smooth, we have that $w_{\Omega,p}$ is continuous up to the boundary (see for example \cite[Corollary 4.2]{Tr}). It is well-known that if we extend $w_{\Omega,p}$ by zero outside $\Omega$, we get that this extension weakly verifies
\[
-\Delta_p w_{\Omega,p}\le 1,\qquad \mbox{ in }\mathbb{R}^N,
\]
i.e.
\begin{equation}
\label{subsolve}
\int_{\mathbb{R}^N} \langle |\nabla w_{\Omega,p}|^{p-2}\,\nabla w_{\Omega,p},\nabla \varphi\rangle\,dx\le \int_{\mathbb{R}^N} \varphi\,dx,\qquad \mbox{ for every } \varphi\in W^{1,p}_0(\mathbb{R}^N), \, \varphi\ge 0.
\end{equation}
Accordingly, for\footnote{Observe that if $1<p<2\le N$, then 
\[
q<\frac{N}{N-1}<\frac{N\,p}{N-p}=p^*.
\]} $1<q<N/(N-1)$ by Proposition \ref{prop:Linftybound} it satisfies the following $L^\infty-L^q$ local bound 
\begin{equation}
\label{localbound}
\|w_{\Omega,p}\|_{L^\infty(Q_{R/2}(x_0))}\le \mathcal{C}_{N,p,q}\,\left[\left(\fint_{Q_{R}(x_0)} (w_{\Omega,p})^q\,dx\right)^\frac{1}{q}+R^\frac{p}{p-1}\right],
\end{equation}
for every cube
\[
Q_R(x_0)=\prod_{i=1}^N (x_0^i-R,x_0^i+R),\qquad \mbox{ with } x_0=(x_0^1,\dots,x_0^N).
\]
We now set
\[
M=\|w_{\Omega,p}\|_{L^\infty(\mathbb{R}^N)}=\|w_{\Omega,p}\|_{L^\infty(\Omega)}.
\] 
We observe that there exists a point $\overline{x}\in \Omega$ such that  
\[
M=w_{\Omega,p}(\overline{x}).
\]
We fix such a point $\overline{x}$ and, without loss of generality, we can assume that it coincides with the origin and we will omit to indicate it.
\par
Finally, we take $\eta$ to be a Lipschitz cut-off function such that 
\[
0\le \eta\le 1,\qquad \eta\equiv 1 \mbox{ on } Q_{R/2},\qquad \eta\equiv 0 \mbox{ on } \mathbb{R}^N\setminus Q_R,
\]
and
\[
\|\nabla \eta\|_{L^\infty}\le \frac{2}{R},
\]
for a radius $R>0$ whose choice will be declared in a while.
We use $\varphi=\eta\,w_{\Omega,p}/\|\eta\,w_{\Omega,p}\|_{L^q(\Omega)}$ as a test function in the definition of $\lambda_{p,q}(\Omega)$. This yields
\[
\Big(\lambda_{p,q}(\Omega)\Big)^\frac{1}{p}\le \frac{\displaystyle\left(\int_{Q_{R}} |\nabla w_{\Omega,p}|^p\,\eta^p\,dx\right)^\frac{1}{p}+\displaystyle\left(\int_{Q_R} |\nabla \eta|^p\,(w_{\Omega,p})^p\,dx\right)^\frac{1}{p}}{\displaystyle\left(\int_{Q_R} |\eta\,w_{\Omega,p}|^q\,dx \right)^\frac{1}{q}}.
\]
We now observe that 
\[
\left(\int_{Q_R} |\nabla \eta|^p\,(w_{\Omega,p})^p\,dx\right)^\frac{1}{p}\le 2^{1+\frac{N}{p}}\,M\,R^\frac{N-p}{p},
\]
thanks to the properties of $\eta$.
For the first term in the numerator, we use the following Caccioppoli inequality
\begin{equation}
\label{caccioppoli}
\int_{Q_R} |\nabla w_{\Omega,p}|^p\,\eta^p\,dx\le p^p\,\int_{Q_R} |\nabla \eta|^p\,(w_{\Omega,p})^p\,dx+p\,\int_{Q_R} \eta^p\,w_{\Omega,p}\,dx.
\end{equation}
This easily follows from \eqref{subsolve}, by inserting the test function $\varphi=\eta^p\,w_{\Omega,p}$. Indeed, with such a choice we get
\[
\int_{Q_R} |\nabla w_{\Omega,p}|^p\,\eta^p\,dx+p\,\int_{Q_R} \langle |\nabla w_{\Omega,p}|^{p-2}\,\nabla w_{\Omega,p},\nabla \eta\rangle \,\eta^{p-1}\,w_{\Omega,p}\,dx\le \int_{Q_R} \eta^p\,w_{\Omega,p}\,dx.
\]
On the left-hand side, by using Young's inequality we get for every $\delta>0$
\[
\begin{split}
p\,\int_{Q_R} \langle |\nabla w_{\Omega,p}|^{p-2}\,\nabla w_{\Omega,p},\nabla \eta,\rangle \eta^{p-1}\,w_{\Omega,p}\,dx&\ge -\delta\,(p-1)\,\int_{Q_R} |\nabla w_{\Omega,p}|^p\,\eta^p\,dx\\
&-\delta^{1-p}\,\int_{Q_R} |\nabla \eta|^p\,(w_{\Omega,p})^p\,dx.
\end{split}
\]
The last two equations in display give
\[
(1-\delta\,(p-1))\,\int_{Q_R} |\nabla w_{\Omega,p}|^p\,\eta^p\,dx\le \delta^{1-p}\,\int_{Q_R} |\nabla \eta|^p\,(w_{\Omega,p})^p\,dx+\int_{Q_R)} \eta^p\,w_{\Omega,p}\,dx.
\]
By choosing $\delta=1/p$, we then obtain \eqref{caccioppoli}, as claimed. In turn, from \eqref{caccioppoli} and the properties of $\eta$, we get
\[
\begin{split}
\left(\int_{Q_R} |\nabla w_{\Omega,p}|^p\,\eta^p\,dx\right)^\frac{1}{p}&\le \left(2^{p+N}\,p^p\,R^{N-p}\,M^p+2^N\,p\,M\,R^N\right)^\frac{1}{p}\\
&=R^\frac{N-p}{p}\,\left(2^{p+N}\,p^p\,M^p+2^N\,p\,M\,R^p\right)^\frac{1}{p}.
\end{split}
\]
Finally, for the denominator we use that
\[
\begin{split}
\left(\int_{Q_R} |\eta\,w_{\Omega,p}|^q\,dx \right)^\frac{1}{q}&\ge \left(\int_{Q_{R/2}} (w_{\Omega,p})^q\,dx \right)^\frac{1}{q}\\
&\ge \frac{1}{\mathcal{C}_{N,p,q}}\,M\,R^\frac{N}{q}- R^\frac{p}{p-1}\,R^\frac{N}{q},
\end{split}
\]
thanks to \eqref{localbound} with $R/2$ in place of $R$ and thanks to the choice of the cube, which is centered at the maximum point of $w_{\Omega,p}$. By collecting all the estimates, we obtained
\[
\Big(\lambda_{p,q}(\Omega)\Big)^\frac{1}{p}\le R^{\frac{N-p}{p}-\frac{N}{q}}\,\frac{\left(2^{p+N}\,p^p\,M^p+2^N\,p\,M\,R^p\right)^\frac{1}{p}}{\displaystyle \frac{1}{\mathcal{C}_{N,p,q}}\,M-\,R^\frac{p}{p-1}}.
\]
It is now time to declare the choice of $R$: we choose it in such a way that
\[
\frac{M}{\mathcal{C}_{N,p,q}}-\,R^\frac{p}{p-1}=\theta\,\frac{M}{\mathcal{C}_{N,p,q}},
\]
for $0<\theta<1$,
that is
\[
R=\left((1-\theta)\,\frac{M}{\mathcal{C}_{N,p,q}}\right)^\frac{p-1}{p}.
\]
This finally gives
\[
\Big(\lambda_{p,q}(\Omega)\Big)^\frac{1}{p}\le \left(\left((1-\theta)\,\frac{M}{\mathcal{C}_{N,p,q}}\right)^\frac{p-1}{p}\right)^{\frac{N-p}{p}-\frac{N}{q}}\,\frac{\left(2^{p+N}\,p^p+2^N\,p\,\left(\dfrac{1-\theta}{\mathcal{C}_{N,p,q}}\right)^{p-1}\right)^\frac{1}{p}}{\displaystyle \frac{\theta}{\mathcal{C}_{N,p,q}}}.
\]
We can now observe that 
\[
M=\|w_{\Omega,p}\|_{L^\infty(\Omega)}\ge \left(\frac{1}{\lambda_p(\Omega)}\right)^\frac{1}{p-1},
\]
thanks to \cite[Theorem 1.3]{BR} (see also \cite[Proposition 6]{BE}). By using this estimate and the fact that
\[
\frac{N-p}{p}-\frac{N}{q}<0,
\]
we get
\[
\Big(\lambda_{p,q}(\Omega)\Big)^\frac{1}{p}\le \Theta_{N,p,q,\theta}\,\Big(\lambda_p(\Omega)\Big)^{\frac{1}{p}\,\left(\frac{N}{q}-\frac{N-p}{p}\right)},
\]
where 
\[
 \Theta_{N,p,q,\theta}=\left(\left(\frac{1-\theta}{\mathcal{C}_{N,p,q}}\right)^\frac{p-1}{p}\right)^{\frac{N-p}{p}-\frac{N}{q}}\,\frac{\left(2^{p+N}\,p^p+2^N\,p\,\left(\dfrac{1-\theta}{\mathcal{C}_{N,p,q}}\right)^{p-1}\right)^\frac{1}{p}}{\displaystyle \frac{\theta}{\mathcal{C}_{N,p,q}}},
\]
for every $0<\theta<1$.
By taking the limit as $p$ goes to $1$ and using Lemma \ref{lemma:limitescemo}, Lemma \ref{lemma:angeo} and Proposition \ref{prop:Linftybound}, we get
\[
h_q(\Omega)\le \Theta_{N,1,q,\theta}\,\Big(h_1(\Omega)\Big)^{\frac{N}{q}-(N-1)},
\]
with
\[
\Theta_{N,1,q,\theta}=\frac{3\cdot 2^N\,\mathcal{C}_{N,1,q}}{\theta}.
\]
We can finally let $\theta$ go to $1$ and obtain the claimed estimate, under the assumption that $\Omega$ is bounded, with constant given by
\[
C=3\cdot 2^N\,\mathcal{C}_{N,1,q}.
\]
\vskip.2cm\noindent
Finally, in order to remove the boundedness and smoothness assumption on $\Omega$, it is sufficient to take an exhaustion $\{\Omega_n\}_{n\in\mathbb{N}}$ of $\Omega$ made of open bounded smooth sets, see \cite[Proposition 8.2.1]{Dan} for the existence of such an exhaustion.
By using that 
\[
\lim_{n\to\infty} h_q(\Omega_n)=h_q(\Omega),\qquad \mbox{ for every } 1\le q<\frac{N}{N-1},
\]
we can pass to the limit in the estimate previously obtained and conclude.
\end{proof}
\begin{remark}
It would be interesting to compute the sharp constant $C$ for the inequality
\[
h_q(\Omega)\le C\,\Big(h_1(\Omega)\Big)^{\frac{N}{q}-(N-1)},\qquad 1<q<\frac{N}{N-1},
\]
among all possible open sets. The constant obtained with the previous proof is probably quite rough: its precise expression can be obtained by looking at the value of $\mathcal{C}_{N,1,q}$ obtained in the proof of Proposition \ref{prop:Linftybound}.
\end{remark}
For $0<q<1$, the previous result holds only partially. More precisely, we have the following
\begin{proposition}[The case $0<q<1$]
\label{prop:subbo}
Let $N\ge 2$, for every $0<q<1$ and every open set $\Omega\subseteq\mathbb{R}^N$ we have
\[
h_q(\Omega)\le \left(N\,\omega_N^\frac{1}{N}\right)^{N-\frac{N}{q}}\,\Big(h_1(\Omega)\Big)^{\frac{N}{q}-(N-1)},
\]
with equality for $N-$dimensional balls. Moreover, there exists an open set $\Omega\subseteq\mathbb{R}^N$ such that
\[
h_1(\Omega)>0\qquad \mbox{ and }\qquad h_q(\Omega)=0.
\]
\end{proposition}
\begin{proof}
The upper bound on $h_q(\Omega)$ can be proved as the lower bound in Theorem \ref{thm:equivalenti}. It is sufficient to observe that in the right-hand side of \eqref{decompose}, the exponent $N-N/q$ is now negative. Thus, by using the Isoperimetric Inequality as above, this time we get an upper bound (still sharp). 
\par
We now take $\Omega=\mathbb{R}^{N-1}\times(-2,2)$. Since $\Omega$ is bounded in the direction $\mathbf{e}_N$, by using that $\lambda_{1,1}(\Omega)=h_1(\Omega)$, we immediately get that $h_1(\Omega)>0$. In order to prove that $h_q(\Omega)=0$, we take the ellipsoid
\[
\mathcal{E}_L=\left\{x=(x_1,\dots,x_N)\in\mathbb{R}^N\, :\, \sum_{i=1}^{N-1}\, \frac{x_i^2}{L^2}+x_N^2<1\right\},
\]
for $L>1$.
Its volume is simply given by $|\mathcal{E}_L|=\omega_N\,L^{N-1}$. The surface measure of its boundary is given by
\[
\mathcal{H}^{N-1}(\partial \mathcal{E}_L)=2\,\int_{B'_L(0)} \sqrt{1+|\nabla f(x')|^2}\,dx', 
\]
where $x'=(x_1,\dots,x_{N-1})$ and $B'_L(0)$ is the $(N-1)-$dimensional ball centered at the origin, having radius $L$. The function $f$ is given by
\[
f(x')=\sqrt{1-\frac{|x'|^2}{L^2}},\qquad \mbox{ for } |x'|<L.
\]
With simple computations, we see that 
\[
\begin{split}
\mathcal{H}^{N-1}(\partial \mathcal{E}_L)=2\,\int_{B'_L(0)} \sqrt{1+|\nabla f(x')|^2}\,dx'&=2\,\int_{B_L'(0)} \sqrt{1+\frac{1}{L^2}\,\frac{|x'|^2}{L^2-|x'|^2}}\,dx'\\
&=2\,(N-1)\,\omega_{N-1}\,\int_0^L \sqrt{1+\frac{1}{L^2}\,\frac{\varrho^2}{L^2-\varrho^2}}\,\varrho^{N-2}\,d\varrho\\
&=2\,(N-1)\,\omega_{N-1}\,L^{N-1}\,\int_0^1 \sqrt{1+\frac{1}{L^2}\,\frac{t^2}{1-t^2}}\,t^{N-2}\,dt\\
&\le 2\,(N-1)\,\omega_{N-1}\,L^{N-1}\,\int_0^1 \sqrt{\frac{1}{1-t^2}}\,t^{N-2}\,dt.
\end{split}
\]
We thus get
\[
h_q(\Omega)\le \frac{\mathcal{H}^{N-1}(\partial \mathcal{E}_L)}{|\mathcal{E}_L|^\frac{1}{q}}=C_{N,q}\,L^{(N-1)\,\left(1-\frac{1}{q}\right)}.
\]
By taking the limit as $L$ goes to $+\infty$ and using that $0<q<1$, we get the desired conclusion.
\end{proof}

\section{Generalized Cheeger's constants for convex planar sets}
\label{sec:5}

\subsection{Unbounded convex sets}
For an open set $\Omega\subsetneq\mathbb{R}^N$, we will use the following notation
\[
d_\Omega(x)=\min_{y\in\partial\Omega} |x-y|,\qquad \mbox{ for } x\in\Omega,
\]
for the {\it distance function}. We also introduce the notation
\[
r_\Omega=\sup_{x\in\Omega} d_\Omega(x).
\]
We recall that such a quantity is called {\it inradius}. This coincides with the supremum of the radii of open balls entirely contained in $\Omega$. Finally, we define
\[
\mathcal{M}(\Omega):=\Big\{x\in\Omega\, :\, B_{r_\Omega}(x)\subseteq\Omega\Big\},
\]
the {\it high ridge set} of $\Omega$. 
\par
The following technical result will be useful in a while. This should be well-known. 
\begin{lemma}
\label{lemma:palleinfinite}
Let $\Omega\subseteq\mathbb{R}^N$ an unbounded open convex set. For every $0<r<r_\Omega$, 
there exists a sequence of points $\{x_n\}_{n\in\mathbb{N}}\subseteq\Omega$, all laying on the same line, such that
\[
\lim_{n\to\infty} |x_n|=+\infty\qquad \mbox{ and }\qquad B_r(x_n)\subseteq \Omega.
\]
Moreover, if $\mathcal{M}(\Omega)\not=\emptyset$, the previous result is valid for $r=r_\Omega$, as well.
\end{lemma}
\begin{proof}
Let us take $0<r<r_\Omega$, by definition there exists a point $\overline{x}\in\Omega$ such that $B_r(\overline{x})\subseteq\Omega$.
For every $\omega\in \mathbb{S}^{N-1}$ we set
\[
L_\omega(\overline x)=\Big\{x\in\mathbb{R}^N\, :\, x=\overline x+t\,\omega \mbox{ for some } t\ge 0\Big\},
\]
i.e. the half-line originating from $\overline x$, with direction $\omega$. By convexity, we have that $L_\omega(\overline x)\cap \Omega$ is a segment. Since $\Omega$ is unbounded, there exists a direction $\omega_0\in\mathbb{S}^{N-1}$ such that 
\[
L_{\omega_0}(\overline x)\cap \Omega=L_{\omega_0}(\overline x).
\]
We now choose the following sequence of points 
\[
x_n=\overline{x}+n\,r\,\omega_0\in \Omega,\qquad \mbox{ for every } n\in\mathbb{N}.
\]
We have to prove that $B_r(x_n)\subseteq\Omega$: at this aim, we are going to show that $d_\Omega(x_n)\ge r$. We fix $n\in\mathbb{N}\setminus\{0\}$ and take $m\ge n+1$. We then consider the convex hull $T_m$ of the disk $B_{r}(\overline x)$ and the point $x_m$. Of course, we have that $T_m\subseteq \Omega$, by convexity of the latter. We then have 
\[
d_\Omega(x_n)\ge \mathrm{dist}(x_n,\partial T_m).
\]
\begin{figure}
\includegraphics[scale=.25]{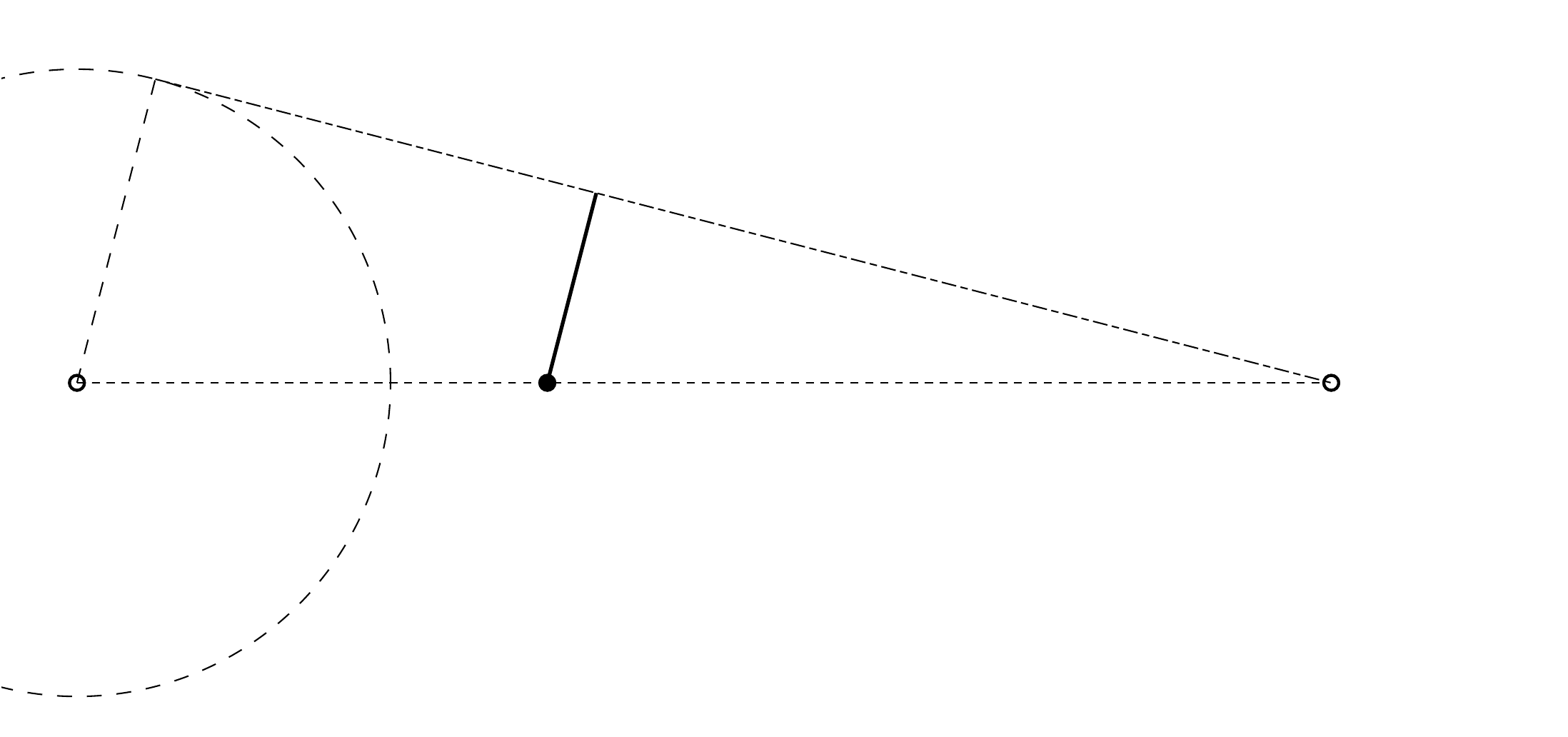}
\caption{The construction in Lemma \ref{lemma:palleinfinite}: the leftmost point is $\overline{x}$, while the rightmost one is $x_{m}$, with $m\ge n+1$. In dashed line, the radius $r$. The black dot is $x_n$, whose distance from the boundary $\partial\Omega$ is at least the length of the segment in bold line.}
\label{fig:lemma51}
\end{figure}
By simple geometric considerations (see Figure \ref{fig:lemma51}), we see that 
\[
\mathrm{dist}(x_n,\partial T_m)\ge \left(1-\frac{n}{m}\right)\,r,\qquad \mbox{ for } m\ge n+1.
\]
By joining the last two estimates and taking the limit as $m$ goes to $\infty$, we get the conclusion.
\end{proof}
In the next geometric results, we restrict ourselves to dimension $N=2$ and we study the properties of unbounded convex sets, with finite inradius. In higher dimension, the picture would be slightly more complicate. It is useful to recall that for a convex set, the high ridge set $\mathcal{M}(\Omega)$, if it is not empty, is a closed convex set, with empty interior. Actually, it coincides with the set of maximum points of the distance function
\[
d_\Omega(x)=\min_{y\in\partial \Omega} |x-y|,\qquad \mbox{ for } x\in\Omega,
\]
which is concave, by convexity of $\Omega$.
\begin{lemma}
\label{lemma:tubi}
Let $\Omega\subseteq\mathbb{R}^2$ be an unbounded open convex set, with $r_\Omega<+\infty$. Let us suppose that $\mathcal{M}(\Omega)\not=\emptyset$. Then, up to a rigid movement, it holds
\begin{equation}
\label{salsiccia}
[0,+\infty)\times(-r_\Omega,r_\Omega)\subseteq \Omega\subseteq \mathbb{R}\times(-r_\Omega,r_\Omega).
\end{equation}
\end{lemma}
\begin{proof}
We first observe that in this case, since $\Omega$ is unbounded, we must have that $\mathcal{M}(\Omega)$ is unbounded as well. This is a plain consequence of Lemma \ref{lemma:palleinfinite} applied with $r=r_\Omega$ and the convexity of $\mathcal{M}(\Omega)$. This yields that $\mathcal{M}(\Omega)$ must contain a half-line. Up to a rigid movement, we can suppose that this coincides with $[0,+\infty)\times\{0\}$. 
Since this half-line is made of centers of disks with maximal radius contained in $\Omega$, we also get that 
\[
[0,+\infty)\times(-r_\Omega,r_\Omega)\subseteq \Omega.
\]
Let us now suppose that there exists a point $z\in\Omega\setminus(\mathbb{R}\times(-r_\Omega,r_\Omega))$. By convexity, then $\Omega$ must contain the convex hull of $\{z\}$ and $[0,+\infty)\times(-r_\Omega,r_\Omega)$. It is easy to see that such a set contains a disk with radius strictly larger than $r_\Omega$. This would violate the maximality of $r_\Omega$. In conclusion, we get that 
\[
\Omega\setminus(\mathbb{R}\times(-r_\Omega,r_\Omega))=\emptyset,
\]
and thus \eqref{salsiccia} follows.
\end{proof}
We now inquire about the structure of an unbounded planar convex set with finite inradius, having empty high ridge set. We have the following result, whose proof is lenghty, though elementary.
\begin{lemma}
\label{lemma:unbounded}
Let $\Omega\subseteq\mathbb{R}^2$ be an unbounded open convex set, with $r_\Omega<+\infty$. Let us suppose that $\mathcal{M}(\Omega)=\emptyset$. Then there exists a convex function $f:(-r_\Omega,r_\Omega)\to\mathbb{R}$ with at least one of the two limits
\[
\lim_{t\to (-r_\Omega)^+} f(t),\qquad\lim_{t\to (r_\Omega)^-} f(t),
\]
equal to $+\infty$,
such that, up to a rigid movement, we have
\[
\Omega=\Big\{(x_1,x_2)\in\mathbb{R}^2\,:\, -r_\Omega<x_2<r_\Omega\, \mbox{ and } x_1>f(x_2)\Big\}.
\]
\end{lemma}
\begin{proof}
We set $\mathbb{S}^1_+=\{\omega=(\cos\vartheta,\sin\vartheta)\, :\, \vartheta\in[0,\pi)\}$.
For every $\omega\in\mathbb{S}^1_+$, we define $\Pi_\omega:\mathbb{R}^2\to \langle \omega^\bot\rangle$ the orthogonal projection on $\langle \omega^\bot\rangle$, given by
\[
\Pi_\omega(x)=x-\langle x,\omega\rangle\,\omega,\qquad \mbox{ for every }x\in\mathbb{R}^2.
\]
We need to show at first that, under the standing assumptions on $\Omega$, the following property holds true: 
\begin{equation}
\label{mandacelabuona!}
\mbox{ there exists a unique $\omega_0\in\mathbb{S}^1_+$ such that $\Pi_{\omega_0}(\Omega)$ is bounded}.
\end{equation}
Indeed, assume this were not true, then we would have two possibilities:
\begin{itemize}
\item[(i)] either there exist (at least) two distinct $\omega_1,\omega_2\in\mathbb{S}^1_+$ such that both $\Pi_{\omega_1}(\Omega)$ and $\Pi_{\omega_2}(\Omega)$ are bounded;
\vskip.2cm
\item[(ii)] or the projection $\Pi_\omega(\Omega)$ is unbounded, for every $\omega\in\mathbb{S}^1_+$.
\end{itemize}
In case (i), let us set $\Omega_i=\Pi_{\omega_i}(\Omega)$, for $i=1,2$. By assumption, these are two non-collinear segments. Accordingly, we have that $\Pi_{\omega_i}^{-1}(\Omega_i)$ are two non-parallel strips. By construction, we would have
\[
\Omega\subseteq \Pi_{\omega_1}^{-1}(\Omega_1)\cap \Pi_{\omega_2}^{-1}(\Omega_2),
\]
and the latter is a bounded set. This would contradict the fact that $\Omega$ is unbounded and thus case (i) can not hold.
\par
We now suppose that case (ii) holds. By Lemma \ref{lemma:palleinfinite}, we get in particular that $\Omega$ must contain a half-line $\mathcal{L}_1$. Up to a rigid movement, we can suppose that $\mathcal{L}_1$ has direction $\mathbf{e}_1=(1,0)$. We now consider the projection $\Pi_{\mathbf{e}_1}(\Omega)$, by assumption this is a one-dimensional unbounded convex set, i.e. it contains a half-line. By convexity, this entails that $\Omega$ must contain another half-line $\mathcal{L}_2$, such that $\mathcal{L}_1\not=\mathcal{L}_2$. More precisely, $\mathcal{L}_1$ and $\mathcal{L}_2$ are not parallel, since the image through $\Pi_{\mathbf{e}_1}$ of every half-line parallel to $\mathcal{L}_1$ would be a single point. Since $\Omega$ is convex, it must contain the convex hull of $\mathcal{L}_1\cup \mathcal{L}_2$. The latter contains arbitrarily large disks, since $\mathcal{L}_1$ and $\mathcal{L}_2$ are not parallel. This contradicts the fact that $r_\Omega<+\infty$.
\par
In conclusion, we established the validity of \eqref{mandacelabuona!}. Without loss of generality, we can suppose that $\omega=\mathbf{e}_1$ and that $\Pi_{\mathbf{e}_1}(\Omega)=(-a,a)$, for a suitable $a>0$. Thus, we have that 
\[
\Omega\subseteq \mathbb{R}\times(-a,a).
\] 
Observe that $a\ge r_\Omega$: indeed, for every $\varepsilon>0$, we have that $\Omega$ contains at least a disk of radius $r_\Omega-\varepsilon$. The orthogonal projection of this disk along the direction $\mathbf{e}_1$ is a segment of length $2\,r_\Omega-\varepsilon$, thus $(-a,a)$ must contain a segment having this length. By arbitrariness of $\varepsilon$, we get that $a\ge r_\Omega$, as claimed. 
\par
We have already observed that $\Omega$ must contain at least a half-line $\mathcal{L}_1$. On the other hand, it can not contain a line, otherwise $\Omega$ would coincide with a strip. Since for a strip the high ridge set is not empty, we would get a contradiction with the assumption $\mathcal{M}(\Omega)=\emptyset$. 
\par
For every $x_2\in(-a,a)$, we now define
\[
f(x_2)=\inf\{x_1\, :\, (x_1,x_2)\in\Omega\}\qquad \mbox{ and }\qquad g(x_2)=\sup\{x_1\, :\, (x_1,x_2)\in\Omega\}.
\]
By construction, we have $(f(x_2),g(x_2))\times\{x_2\}\subseteq\Omega$. Since $\Omega$ can not contain a line, for every $x_2\in(-a,a)$ we have that either $f(x_2)>-\infty$ or $g(x_2)<+\infty$. Let us fix $x_2\in(-a,a)$, without loss of generality we suppose that $f(x_2)>-\infty$. Then we must have $g(x_2)=+\infty$: indeed, if $g(x_2)<+\infty$, then $\Omega$ has to contain the convex hull of the half-line $\mathcal{L}_1$ and the segment $(f(x_2),g(x_2))\times\{x_2\}$. Such a convex hull contains in particular the half-line $(f(x_2),+\infty)\times\{x_2\}$, thus contradicting the definition of $g(x_2)$.
\par
We pick $t\in(-a,a)$ such that $t\not=x_2$, we show that we must have $f(t)>-\infty$ and $g(t)=+\infty$, as well. We already know that these quantities can not be both infinite or both finite. We have to exclude the case $f(t)=-\infty$ and $g(t)<+\infty$. In this case, $\Omega$ would contain the convex hull of the two half-lines $(-\infty,g(t))\times\{t\}$ and $(f(x_2),+\infty)\times\{x_2\}$. Such a convex hull is given by a strip, thus $\Omega$ would contain a line, which is not possible. 
\par
From this discussion, we finally obtain that
\[
f(x_2)>-\infty \qquad \mbox{ and }\qquad g(x_2)=+\infty,\qquad \mbox{ for every } -a<x_2<a.
\] 
It is not difficult to see that $f$ is convex: let us take $\tau\in(0,1)$ and $x_2,t_2\in(-a,a)$. Let $x_1$ and $t_1$ be such that $(x_1,x_2)\in\Omega$ and $(t_1,t_2)\in\Omega$. By convexity, we have $(\tau\,x_1+(1-\tau)\,t_1,\tau\,x_2+(1-\tau)\,t_2)\in\Omega$, as well. Thus, by definition of $f$ we get
\[
f(\tau\,x_2+(1-\tau)\,t_2)\le \tau\,x_1+(1-\tau)\,t_1.
\]
By taking first the infimum over the admissible $x_1$ and then over the admissible $t_1$, we finally obtain
\[
f(\tau\,x_2+(1-\tau)\,t_2)\le \tau\,f(x_2)+(1-\tau)\,f(t_2).
\]
The fact that $\Omega$ coincides with the epigraph of $f$ follows by its construction and the convexity of $\Omega$. We also observe that actually it holds $a=r_\Omega$. Indeed, we already proved that $a\ge r_\Omega$: on the other hand, if $a>r_\Omega$ by the previous properties we would get that $\Omega$ contains the two half-lines
\[
(f(r_\Omega),+\infty)\times\{r_\Omega\}\qquad \mbox{ and }\qquad (f(-r_\Omega),+\infty)\times\{-r_\Omega\}.
\]
Then it must contain their convex hull, as well. In particular, $\Omega$ would contain an open disk of radius $r_\Omega$, violating the fact that $\mathcal{M}(\Omega)=\emptyset$. 
\par
Finally, the fact that $f$ must blow-up in at least one of the extrema $r_\Omega$ or $-r_\Omega$ follows from a similar argument: if both limits were finite, then $\Omega$ would contain an open half-strip with width $2\,r_\Omega$. In particular, we would have again $\mathcal{M}(\Omega)\not=\emptyset$.
\end{proof}

\subsection{An existence result}

The main result of this section is the following
\begin{theorem}
\label{thm:existence_convex}
Let $\Omega\subseteq\mathbb{R}^2$ be an open convex set with $r_\Omega<+\infty$. Then, for every $1<q<2$ we have 
\begin{equation}
\label{relaxed}
h_q(\Omega)=\inf\left\{\frac{\mathcal{H}^{1}(\partial E)}{|E|^\frac{1}{q}}\, :\, E\subseteq\Omega \mbox{ is a bounded open convex set with }|E|>0\right\}.
\end{equation}
Moreover,
the infimum is attained if and only if 
\[
\mathcal{M}(\Omega):=\Big\{x\in\Omega\, :\, B_{r_\Omega}(x)\subseteq\Omega\Big\}\not=\emptyset.
\]
\end{theorem}
\begin{proof}
We divide the proof in 4 parts, for ease of readability.
\vskip.2cm\noindent
{\bf Part 1: reduction to connected sets.} Let us take an open set $E\Subset \Omega$ with smooth boundary and let us suppose that $E=E_1 \cup E_2$, with $E_1$ and $E_2$ disjoint. Without loss of generality, we can suppose that 
\[
\frac{\mathcal{H}^{1}(\partial E_1)}{|E_1|^\frac{1}{q}}\le \frac{\mathcal{H}^{1}(\partial E_2)}{|E_2|^\frac{1}{q}}.
\]
By using the subadditivity of the concave map $\tau\mapsto \tau^{1/q}$ and Lemma \ref{lemma:ulanbator} with $\beta=1$, we have
\[
\begin{split}
\frac{\mathcal{H}^{1}(\partial E_1)+\mathcal{H}^{1}(\partial E_2)}{\left(|E_1|+|E_2|\right)^\frac{1}{q}}&>\frac{\mathcal{H}^{1}(\partial E_1)+\mathcal{H}^{1}(\partial E_2)}{|E_1|^\frac{1}{q}+|E_2|^\frac{1}{q}}\\
&>\min\left\{\frac{\mathcal{H}^{1}(\partial E_1)}{|E_1|^\frac{1}{q}},\frac{\mathcal{H}^{1}(\partial E_2)}{|E_2|^\frac{1}{q}}\right\} =\frac{\mathcal{H}^{1}(\partial E_1)}{|E_1|^\frac{1}{q}}.
\end{split}
\]
This estimate guarantees that we can restrict the minimization to {\it connected} open sets with smooth boundary $E\Subset\Omega$.
\vskip.2cm\noindent
{\bf Part 2: reduction to convex sets.} As already observed in \cite[Theorem 1.2]{PS}, if we take an open connected set $E\Subset \Omega$ with smooth boundary, then the convex hull $E^{\rm ch}$ of $E$ is such that
\[
|E^{\rm ch}|\ge |E|\qquad \mbox{ and }\qquad \mathcal{H}^{1}(\partial E^{\rm ch})\le \mathcal{H}^{1}(\partial E).
\]
In the latter, we crucially used that we are in dimension $N=2$. See for example \cite[Proposition 5]{FF} for a proof of this result in the wider context of finite perimeter sets. This shows that 
\[
h_q(\Omega)\ge \inf\left\{\frac{\mathcal{H}^{1}(\partial E)}{|E|^\frac{1}{q}}\, :\, E\subseteq\Omega \mbox{ is a bounded open convex set with }|E|>0\right\}.
\]
In order to prove the reverse equality, we take an admissible convex set $E\subseteq\Omega$. We fix $x_0\in\Omega$ and for every $t\in(0,1)$ define the rescaled set $E_t=t\,(E-x_0)+x_0$. By construction, we have $E_t\Subset \Omega$, for every $t\in (0,1)$.
We now use that for the open convex set $E_t$, there exists a sequence $\{E_{t,n}\}_{n\ge n_0}$ of smooth open convex sets, such that 
\[
\lim_{n\to\infty} |E_{t,n}|=|E_t|\qquad \mbox{ and }\qquad \lim_{n\to\infty} \mathcal{H}^1(\partial E_{t,n})=\mathcal{H}^1(\partial E_t).
\]
Moreover, since $E_t\Subset\Omega$, this sequence can be constructed so that $E_{t,n}\Subset\Omega$ for $n$ large enough (depending on $t$), see Lemma \ref{lm:approximation} below for these properties, for example. Thus, we obtain
\[
h_q(\Omega)\le \lim_{n\to\infty}\frac{\mathcal{H}^{1}(\partial E_{t,n})}{|E_{t,n}|^\frac{1}{q}}=\frac{\mathcal{H}^{1}(\partial E_t)}{|E_t|^\frac{1}{q}}=t^{1-\frac{2}{q}}\,\frac{\mathcal{H}^{1}(\partial E)}{|E|^\frac{1}{q}}.
\]
By arbitrariness of both $0<t<1$ and $E$, this is enough to get \eqref{relaxed}.
\vskip.2cm\noindent
\vskip.2cm\noindent
{\bf Part 3: existence.} We do not assume that $\Omega$ is bounded, but only that
$\Omega$ contains at least a ball of maximal radius $r_\Omega$. We recall that $\mathcal{M}(\Omega)$ is a convex closed set, with empty interior. 
We have two distinguish two cases:
\begin{itemize}
\item[(i)] either $\mathcal{M}(\Omega)$ is bounded;
\vskip.2cm
\item[ (ii)] or $\mathcal{M}(\Omega)$ is unbounded.
\end{itemize}
In case (i), it is not difficult to see that $\Omega$ must be bounded, as well. Indeed, suppose by contradiction that $\Omega$ is unbounded. From Lemma \ref{lemma:palleinfinite}, we could infer existence of a infinite sequence of disks $B_{r_\Omega}(x_n)$ with centers diverging at infinity, all contained in $\Omega$. Thus, $\{x_n\}_{n\in\mathbb{N}}\subseteq\mathcal{M}(\Omega)$ and this would violate the boundedness of $\mathcal{M}(\Omega)$.
\par
If $\Omega$ is bounded, existence can be proved by using the Direct Method in a standard way. We take $\{E_n\}_{n\in\mathbb{N}}$ a minimizing sequence made of convex sets contained in $\Omega$, such that $|E_n|>0$ for every $n\in\mathbb{N}$. We can assume that 
\begin{equation}
\label{minimizing}
\frac{\mathcal{H}^{1}(\partial E_n)}{|E_n|^\frac{1}{q}}<h_q(\Omega)+\frac{1}{n+1},\qquad \mbox{ for every } n\in\mathbb{N}. 
\end{equation}
In particular, the leftmost quantity is uniformly bounded by $h_q(\Omega)+1$. By applying the Isoperimetric Inequality to each $E_n$, we get
\begin{equation}
\label{lowerbound_mass}
\frac{2\,\sqrt{\pi}}{h_q(\Omega)+1}<|E_n|^{\frac{1}{q}-\frac{1}{2}},\qquad \mbox{ for every } n\in\mathbb{N}.
\end{equation}
Thanks to the fact that $q<2$, this gives a uniform lower bound on $|E_n|$. 
We now observe that, from the sequence of open sets $\{E_n\}_{n\in\mathbb{N}}$ contained in the compact set $\overline\Omega$, we can extract a subsequence (not relabeled) that converges with respect to the Hausdorff complementary topology\footnote{We recall that this means that
\[
\lim_{n\to\infty} \max\left\{\max_{x\in \overline\Omega\setminus E_n} \mathrm{dist}(x,\overline\Omega\setminus E),\,\max_{x\in \overline\Omega\setminus E} \mathrm{dist}(x,\overline\Omega\setminus E_n)\right\}=0,
\]
see for example \cite[D\'efinition 2.2.7]{HP}.} to an open set $E\subseteq \overline\Omega$ (see \cite[Corollaire 2.2.24]{HP}). 
Moreover, by \cite[point 8, page 33]{HP} we know that $E$ is still convex. We also observe that for convex sets, the Hausdorff convergence implies the convergence in the sense of characteristic functions, as well. This means that we have 
\[
\lim_{n\to\infty} \|1_{E_n}-1_E\|_{L^1(\Omega)}=0,
\] 
and thus by \cite[Proposition 2.3.6]{HP} we get 
\[
\lim_{n\to\infty} |E_n|=|E|\qquad \mbox{ and }\qquad \mathcal{H}^1(\partial E)=|\nabla 1_{E}|(\mathbb{R}^N)\le \liminf_{n\to\infty} |\nabla 1_{E_n}|(\mathbb{R}^N)=\liminf_{n\to\infty} \mathcal{H}^{1}(\partial E_n).
\]
The first fact, in conjunction with \eqref{lowerbound_mass}, shows in particular that $|E|>0$. From \eqref{minimizing} and the previous limits, we eventually get that $E$ must be a minimizer. 
\vskip.2cm\noindent
Let us consider the case (ii). In this case $\Omega$ is unbounded, as well. By Lemma \ref{lemma:tubi}, we know that 
\[
[0,+\infty)\times(-r_\Omega,r_\Omega)\subseteq \Omega\subseteq \mathbb{R}\times(-r_\Omega,r_\Omega),
\]
up to a rigid movement.
Once we gained this geometric information on $\Omega$, we can proceed to prove existence of an optimal set. We start by taking a minimizing sequence $\{E_n\}_{n\in\mathbb{N}}$ of open convex sets as above. We still have the bounds \eqref{minimizing} and \eqref{lowerbound_mass}. 
We need to infer some compactness on $\{E_n\}_{n\in\mathbb{N}}$, by paying attention to the fact that now $\Omega$ is unbounded. Here we will crucially exploit that $1<q<2$. Indeed, this assumption enables us to use the following ``four terms geometric inequality''
\[
C_1\,\mathrm{diam}(E_n)\le \Big(r_{E_n}\Big)^\frac{1}{q-1}\,\left(\frac{\mathcal{H}^1(\partial E_n)}{|E_n|^\frac{1}{q}}\right)^\frac{q}{q-1},
\]
for a constant $C_1=C_1(q)>0$ which degenerates as $q\nearrow 2$, see \cite[Proposition B.6]{Bra}. By using that $r_{E_n}\le r_\Omega<+\infty$, the fact that $1<q<2$ and \eqref{minimizing}, we finally obtain that 
\[
\mathrm{diam}(E_n)\le C_2,\qquad \mbox{ for every } n\in\mathbb{N},
\]
with $C_2$ not depending on $n$. In particular, by \eqref{salsiccia}, we get that $\{E_n\}_{n\in\mathbb{N}}$ is a sequence of convex sets with equi-bounded diameters contained in the strip $\mathbb{R}\times(-r_\Omega,r_\Omega)$. Thus, for every $n\in\mathbb{N}$, there exists a translated copy $\widetilde{E}_n$ of $E_n$ such that
\[
\widetilde{E}_n\subseteq (0,C_2)\times(-r_\Omega,r_\Omega)\subseteq \Omega,
\]
where we used again \eqref{salsiccia}. We can now repeat the argument of the case {\it (i)} above, applied to the sequence $\{\widetilde{E}_n\}_{n\in\mathbb{N}}$, to infer existence.
\vskip.2cm\noindent
{\bf Step 4: non-existence.} We are left with proving that if $\mathcal{M}(\Omega)=\emptyset$, then the minimization problem does not admit a solution. We first observe that this condition assures that $\Omega$ is unbounded. According to Lemma \ref{lemma:unbounded}, the set $\Omega$ has a very peculiar form: up to a rigid movement, it is given by
\[
\Omega=\Big\{(x_1,x_2)\in\mathbb{R}^2\,:\, -r_\Omega<x_2<r_\Omega\, \mbox{ and } x_1>f(x_2)\Big\},
\]
for some convex function $f:(-r_\Omega,r_\Omega)\to\mathbb{R}$, which blows-up to $+\infty$ at (at least) one of the two boundary points $x_2=r_\Omega$ or $x_2=-r_\Omega$.  Thus, if $E\subseteq\Omega$ would be an optimal set, we could translate it rightward and slightly scale it by a factor $t>1$, without exiting from $\Omega$ (see Figure \ref{fig:strisciona}). 
\begin{figure}
\includegraphics[scale=.25]{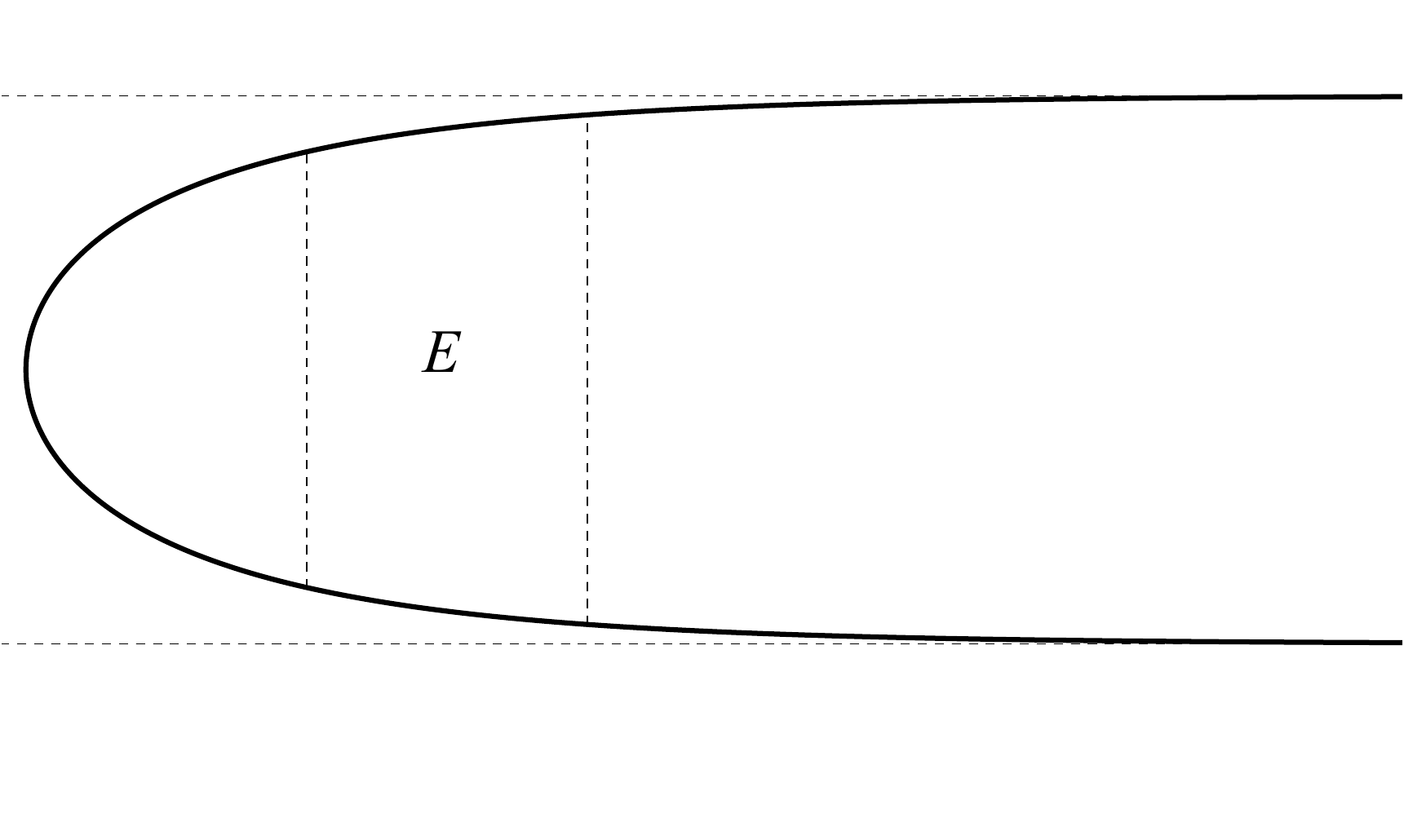}
\caption{An unbounded convex sets with finite inradius, such that the high ridge set $\mathcal{M}(\Omega)$ is empty. The subset $E\subseteq\Omega$ can be moved ``rightward'' and then ``inflated'', without exiting from $\Omega$.}
\label{fig:strisciona}
\end{figure}
By the scaling properties of the ratio $\mathcal{H}^1(\cdot)/|\cdot|^{1/q}$, this would violate the minimality of $E$. 
\end{proof}

\begin{remark}[Generalized principal frequencies]
\label{rem:simon}
The previous existence result for $h_q(\Omega)=\lambda_{1,q}(\Omega)$ is a bit surprising, since an analogous statement
{\it does not} hold for 
\[
\lambda_{p,q}(\Omega)=\inf_{u\in C^\infty_0(\Omega)}\left\{\int_\Omega |\nabla u|^p\,dx\, :\, \int_\Omega |u|^q=1\right\}=\inf_{u\in W^{1,p}_0(\Omega)}\left\{\int_\Omega |\nabla u|^p\,dx\, :\, \int_\Omega |u|^q=1\right\},
\]
with $1<p<q$. Indeed, let us consider the half-strip
\[
\Omega=(0,+\infty)\times(-1,1),
\]
for which we have $\mathcal{M}(\Omega)\not=\emptyset$.
We can show that this time the infimum value $\lambda_{p,q}(\Omega)$ {\it is not attained} in $W^{1,p}_0(\Omega)$. To see this, it is sufficient to notice that 
\[
\lambda_{p,q}(\Omega)=\lambda_{p,q}(\Omega+a\,\mathbf{e}_1),
\]
for every $a\in\mathbb{R}$, since sharp Poincar\'e-Sobolev constants are invariant by translations. On the other hand, for $a>0$ we have 
\[
\Omega+a\,\mathbf{e}_1\subsetneq \Omega.
\]
Assume that $u\in W^{1,p}_0(\Omega)$ is a minimizer for $\lambda_{p,q}(\Omega)$. Without loss of generality, we can assume this to be positive. Then 
\[
v(x,y)=u(x,y-a),
\] 
would be a positive minimizer for $\lambda_{p,q}(\Omega+a\,\mathbf{e}_1)$, as well. By extending $v$ to $0$ to $\Omega\setminus (\Omega+a\,\mathbf{e}_1)$, we would get 
\[
v\in W^{1,p}_0(\Omega),
\]
and
\[
\int_{\Omega} |\nabla v|^p\,dx=\lambda_{p,q}(\Omega)\, \left(\int_\Omega |v|^q\,dx\right)^\frac{p}{q}.
\]
This shows that $v$ is a positive minimizer for $\lambda_{p,q}(\Omega)$. By optimality, it must be a weak solution of
\[
-\Delta_p v=\lambda_{p,q}(\Omega)\,\|v\|_{L^q(\Omega)}^{p-q}\,v^{q-1},\qquad \mbox{ in }\Omega.
\]
In particular, $v$ is a weakly $p-$superhamonic function in $\Omega$, not identically vanishing. By the strong minimum principle, we get a contradiction, since $v$ is identically zero on the set
\[
\Omega\setminus (\Omega+a\,\mathbf{e}_1),
\]
which has positive measure.
\end{remark}
We conclude this section, by showing that for $0<q<1$ the situation abruptly changes. Indeed, for unbounded sets the problem $h_q(\Omega)$ {\it never} has a solution.
\begin{lemma}
\label{lemma:appu}
Let $\Omega\subseteq\mathbb{R}^2$ be an unbounded open convex set, with $r_\Omega<+\infty$. If $0<q<1$, then we have 
\[
h_q(\Omega)=0.
\]
\end{lemma}
\begin{proof}
By appealing to Lemma \ref{lemma:palleinfinite}, there exists a sequence of disks $B_{r}(x_n)\subseteq \Omega$ with fixed radius and centers diverging at infinity. In particular, by taking as $E_n$ the convex hull of $B_{r}(x_0)$ and $B_{r}(x_n)$, we get
\[
h_q(\Omega)\le \lim_{n\to\infty}\frac{\mathcal{H}^1(E_n)}{|E_n|^\frac{1}{q}}=\lim_{n\to\infty}\frac{2\,\pi\,r+2\,|x_0-x_n|}{(\pi\,r^2+2\,r\,|x_0-x_n|)^\frac{1}{q}}=0,
\]
thanks to the fact that $1/q>1$. Thus, if $\Omega$ is an unbounded open set, we have $h_q(\Omega)=0$. Accordingly, we can not have existence of an optimal set.
\end{proof}

\appendix

\section{An a priori local $L^\infty$ bound}
\label{sec:A}

In what follows, for $1\le p<N$ we will indicate by $T_{N,p}$ the sharp constant in the Sobolev inequality, that is
\[
T_{N,p}=\sup_{\varphi\in C^\infty_0(\mathbb{R}^N)} \Big\{\|\varphi\|_{L^{p^*}(\mathbb{R}^N)}\, :\, \|\nabla \varphi\|_{L^p(\mathbb{R}^N)}=1\Big\}.
\]
Its explicit expression can be found for example in \cite[equation (2)]{Ta}. For our scopes, it is useful to recall that 
\begin{equation}
\label{costante_talenti}
\lim_{p\searrow 1} T_{N,p}=T_{N,1}=\frac{1}{N\,\omega_{N}^\frac{1}{N}},
\end{equation}
where on the right-hand side we can recognize the reciprocal of the sharp Euclidean isoperimetric constant. The following result is well-known, but we need to keep track of the relevant constant, as $p$ goes to $1$.
\begin{proposition}
\label{prop:Linftybound}
Let $1<p<2$ and let $u\in W^{1,p}_{\rm loc}(\mathbb{R}^N)\cap L^\infty_{\rm loc}(\mathbb{R}^N)$ be a non-negative local weak subsolution of 
\[
-\Delta_p u=1,\qquad \mbox{ in }\mathbb{R}^N.
\]
Then, for every $p\le q\le p^*$ there exists a constant $\mathcal{C}_{N,p,q}>0$ such that for every cube $Q_{R_0}(x_0)$ we have 
\[
\|u\|_{L^\infty(Q_{R_0/2}(x_0))}\le \mathcal{C}_{N,p,q}\,\left[\left(\fint_{Q_R(x_0)} |u|^q\,dx\right)^\frac{1}{q}+R_0^\frac{p}{p-1}\right].
\]
Moreover, for every fixed $1<q<N/(N-1)$ the constant $\mathcal{C}_{N,p,q}>0$ has a finite positive limit $\mathcal{C}_{N,1,q}$, as $p$ goes to $1$.
\end{proposition}
\begin{proof}
We will use the standard Moser iteration technique, by paying due attention to the constants appearing in the estimates, exactly as in \cite[Lemma A.1]{BFR}.
We fix $R_0/2\le r<R\le R_0$ and a pair of concentric cubes $Q_r(x_0)\subseteq Q_R(x_0)$. For simplicity, from now on we will omit indicating the center $x_0$. We take $\eta$ to be a standard Lipschitz cut-off function, such that
\[
0\le \eta\le 1,\qquad \eta\equiv 1 \mbox{ on } Q_r,\qquad \eta\equiv 0 \mbox{ on } \mathbb{R}^N\setminus Q_R,
\]
and
\[
\|\nabla \eta\|_{L^\infty}= \frac{1}{R-r}.
\]
We then use the test function
\[
\varphi=\eta^p\,(u+t)^\beta,
\]
where $t>0$ and $\beta\ge 1$. We get
\[
\begin{split}
\beta\,\int \left|\nabla u\right|^p\,(u+t)^{\beta-1}\,\eta^p\,dx&+p\,\int \langle |\nabla u|^{p-2}\,\nabla u,\nabla \eta\rangle\,(u+t)^{\beta-1}\,\eta^{p-1}\,dx\\
&\le \int \eta^p\,(u+t)^{\beta}\,dx.
\end{split}
\]
Observe that by Young's inequality, we have for every $\delta>0$
\[
\begin{split}
p\,\int \langle |\nabla u|^{p-2}\,\nabla u,\nabla \eta\rangle\,(u+t)^{\beta-1}\,\eta^{p-1}\,dx&\ge -(p-1)\,\delta\,\int |\nabla u|^p\,(u+t)^{\beta-1}\,\eta^p\,dx\\
&-\delta^{1-p}\,\int |\nabla \eta|^p\,(u+t)^{\beta+p-1}\,dx.
\end{split}
\]
In particular, by choosing $\delta=\beta/p$ we get
\[
\begin{split}
\frac{\beta}{p}\,\int \left|\nabla u\right|^p\,(u+t)^{\beta-1}\,\eta^p\,dx&\le \left(\frac{p}{\beta}\right)^{p-1}\,\int |\nabla \eta|^p\,(u+t)^{\beta+p-1}\,dx+\int \eta^p\,(u+t)^{\beta}\,dx.
\end{split}
\]
We now observe that
\[
\left|\nabla u\right|^p\,(u+t)^{\beta-1}=\left(\frac{p}{\beta+p-1}\right)^p\,\left|\nabla (u+t)^\frac{\beta+p-1}{p}\right|^p,
\]
and
\[
(u+t)^{\beta}\le (u+t)^{\beta+p-1}\,\frac{1}{t^{p-1}}.
\]
Thus, from the previous estimate, we obtain
\[
\begin{split}
\int \left|\nabla (u+t)^\frac{\beta+p-1}{p}\right|^p\,\eta^p\,dx&\le \left(\frac{\beta+p-1}{\beta}\right)^p\,\int |\nabla \eta|^p\,(u+t)^{\beta+p-1}\,dx\\
&+\left(\frac{\beta+p-1}{p}\right)^p\,\frac{p}{\beta}\,\frac{1}{t^{p-1}}\,\int \eta^p\,(u+t)^{\beta+p-1}\,dx.
\end{split}
\] 
In order to simplify a bit the estimate, we use that
\[
\left(\frac{\beta+p-1}{\beta}\right)^p\le \left(\frac{\beta+p-1}{p}\right)^{p-1}\,p,
\]
and
\[
\left(\frac{\beta+p-1}{p}\right)^p\,\frac{p}{\beta}\le \left(\frac{\beta+p-1}{p}\right)^{p-1}\,p.
\]
This yields
\begin{equation}
\label{ciaone}
\int \left|\nabla (u+t)^\frac{\beta+p-1}{p}\right|^p\,\eta^p\,dx\le \left(\frac{\beta+p-1}{p}\right)^{p-1}\,p\,\int \left[|\nabla \eta|^p+\frac{\eta^p}{t^{p-1}}\right]\,(u+t)^{\beta+p-1}\,dx.
\end{equation}
By combining Minkowski's inequality and \eqref{ciaone}, we can infer
\[
\begin{split}
\left(\int \left|\nabla \left((u+t)^\frac{p+\beta-1}{p}\,\eta\right)\right|^p\,dx\right)^\frac{1}{p}&\le \left(\int \left|\nabla (u+t)^\frac{\beta+p-1}{p}\right|^p\,\eta^p\,dx\right)^\frac{1}{p}\\
&+\left(\int |\nabla \eta|^p\,(u+t)^{p+\beta-1}\,dx\right)^\frac{1}{p}\\
&\le \left(\frac{\beta+p-1}{p}\right)^\frac{p-1}{p}\,p^\frac{1}{p}\,\left(\int \left[|\nabla \eta|^p+\frac{\eta^p}{t^{p-1}}\right]\,(u+t)^{\beta+p-1}\,dx\right)^\frac{1}{p}\\
&+\left(\int |\nabla \eta|^p\,(u+t)^{p+\beta-1}\,dx\right)^\frac{1}{p}.
\end{split}
\]
We can bound from below the leftmost integral by using Sobolev's inequality. This yields
\[
\begin{split}
\frac{1}{T_{N,p}}&\,\left(\int \left((u+t)^\frac{p+\beta-1}{p}\,\eta\right)^{p^*}\,dx\right)^\frac{1}{p^*}\\
&\le \left(\frac{\beta+p-1}{p}\right)^\frac{p-1}{p}\,p^\frac{1}{p}\,\left(\int \left[|\nabla \eta|^p+\frac{\eta^p}{t^{p-1}}\right]\,(u+t)^{\beta+p-1}\,dx\right)^\frac{1}{p}\\
&+\left(\int |\nabla \eta|^p\,(u+t)^{p+\beta-1}\,dx\right)^\frac{1}{p}.
\end{split}
\]
It is now time to use the properties of $\eta$. These lead us to 
\begin{equation}
\label{quasipronti}
\begin{split}
\frac{1}{T_{N,p}}&\,\left(\int_{Q_r} \left((u+t)^\frac{p+\beta-1}{p}\right)^{p^*}\,dx\right)^\frac{1}{p^*}\\
&\le \left(\frac{\beta+p-1}{p}\right)^\frac{p-1}{p}\,p^\frac{1}{p}\, \left[\frac{1}{(R-r)^p}+\frac{1}{t^{p-1}}\right]^\frac{1}{p}\, \left(\int_{Q_R} (u+t)^{\beta+p-1}\,dx\right)^\frac{1}{p}\\
&+\frac{1}{R-r}\,\left(\int_{Q_R}(u+t)^{p+\beta-1}\,dx\right)^\frac{1}{p}.
\end{split}
\end{equation}
We also observe that 
\[
\frac{1}{R-r}\le \left(\frac{\beta+p-1}{p}\right)^\frac{p-1}{p}\,p^\frac{1}{p}\, \left[\frac{1}{(R-r)^p}+\frac{1}{t^{p-1}}\right]^\frac{1}{p}.
\]
By using this elementary observation, from \eqref{quasipronti} we get
\[
\begin{split}
&\left(\int_{Q_r} \left((u+t)^\frac{p+\beta-1}{p}\right)^{p^*}\,dx\right)^\frac{1}{p^*}\\
&\le 2\,T_{N,p}\,\left(\frac{\beta+p-1}{p}\right)^\frac{p-1}{p}\,p^\frac{1}{p}\, \left[\frac{1}{(R-r)^p}+\frac{1}{t^{p-1}}\right]^\frac{1}{p}\, \left(\int_{Q_R} (u+t)^{\beta+p-1}\,dx\right)^\frac{1}{p}.
\end{split}
\]
We now set $\vartheta=(\beta+p-1)/p$ in the previous estimate and raise both sides to the power $1/\vartheta$. This gives
\begin{equation}
\label{quasiquasipronti}
\begin{split}
\|u+t\|_{L^{p^*\vartheta}(Q_{r})}
&\le \left(\vartheta^\frac{p-1}{p}\right)^\frac{1}{\vartheta}\,(2^p\,p\,T^p_{N,p})^\frac{1}{p\,\vartheta}\, \left[\frac{1}{(R-r)^p}+\frac{1}{t^{p-1}}\right]^\frac{1}{p\,\vartheta}\, \|u+t\|_{L^{p\,\vartheta}(Q_{R})}.
\end{split}
\end{equation}
We want to iterate the estimate \eqref{quasiquasipronti}, on a sequence of shrinking cubes. At this aim, we take $p\le q\le p^*$ and set
\[
\vartheta_0=\frac{q}{p},\qquad\vartheta_{i+1}=\frac{p^*}{p}\,\vartheta_i=\left(\frac{p^*}{p}\right)^{i+1}\,\frac{q}{p},\qquad i\in\mathbb{N},
\]
and 
\[
R_i=\frac{R_0}{2}+\frac{R_0}{2^i},\qquad i\in\mathbb{N},
\]
where $R_0$ has been fixed at the beginning. From \eqref{quasiquasipronti}, we get
\[
\begin{split}
\|u+t\|_{L^{p\,\vartheta_{i+1}}(Q_{R_{i+1}})}&\le \left(\vartheta_i^\frac{p-1}{p}\right)^\frac{1}{\vartheta_i}\,(2^p\,p\,T^p_{N,p})^\frac{1}{p\,\vartheta_i}\\
&\times\, \left[\left(\frac{2^{i+1}}{R_0}\right)^p+\frac{1}{t^{p-1}}\right]^\frac{1}{p\,\vartheta_i}\, \|u+t\|_{L^{p\,\vartheta_i}(Q_{R_i})}.
\end{split}
\]
We now choose the free parameter $t$: we take it to be
\begin{equation}
\label{t}
t=R_0^\frac{p}{p-1}.
\end{equation}
With simple manipulations, we then obtain
\[
\begin{split}
\|u+t\|_{L^{p\,\vartheta_{i+1}}(Q_{R_{i+1}})}&\le \left(\vartheta_i^\frac{p-1}{p}\right)^\frac{1}{\vartheta_i}\,\left(\frac{2^p\,p\,T^p_{N,p}}{R_0^p}\right)^\frac{1}{p\,\vartheta_i}\, 2^{\frac{i+2}{\vartheta_i}}\, \|u+t\|_{L^{p\,\vartheta_i}(Q_{R_i})}.
\end{split}
\]
We start from $i=0$ and iterate infinitely many times this estimate. By using that
\[
\frac{1}{p}\,\sum_{i=0}^\infty \frac{1}{\vartheta_i}=\frac{1}{q}\,\sum_{i=0}^\infty \left(\frac{p}{p^*}\right)^i=\frac{N}{q\,p},
\]
together with
\[
\begin{split}
\prod_{i=0}^\infty\left(\vartheta_i^\frac{p-1}{p}\right)^\frac{1}{\vartheta_i}&=\lim_{n\to\infty} \exp\left(\frac{p-1}{p}\,\sum_{i=0}^n \frac{1}{\vartheta_i}\,\log \vartheta_i\right)\\
&=\lim_{n\to\infty} \exp\left(\frac{p-1}{q}\,\sum_{i=0}^n \left(\frac{p}{p^*}\right)^{i}\,\left[i\,\log \frac{p^*}{p}+\log \frac{q}{p}\right]\right)\\
&=\exp\left(\frac{p-1}{q}\,\frac{N\,(N-p)}{p^2}\,\log\frac{p^*}{p}\right)\,\cdot\exp\left(\frac{N}{p}\,\frac{p-1}{q}\,\log\frac{q}{p}\right)=:\mathcal{A}_{N,p,q},
\end{split}
\]
and
\[
\begin{split}
\prod_{i=0}^\infty 2^{\frac{i+2}{\vartheta_i}}&=\lim_{n\to\infty} \exp\left(\sum_{i=0}^n\frac{i+2}{\vartheta_i}\,\log 2\right)\\
&=\exp\left(\frac{p}{q}\sum_{i=0}^\infty(i+2)\,\left(\frac{p}{p^*}\right)^i\,\log 2\right)\\
&\le \exp\left(\frac{p}{q}\sum_{i=0}^\infty (i+2)\,\left(1-\frac{1}{N}\right)^i\,\log 2\right)=:\mathcal{B}_{N,p,q},
\end{split}
\]
we finally get the estimate
\[
\|u+t\|_{L^\infty(Q_{R_0/2})}\le \mathcal{A}_{N,p,q}\,\mathcal{B}_{N,p,q}\,\frac{\Big(2^{p}\,p\,T^p_{N,p}\Big)^\frac{N}{p\,q}}{R_0^\frac{N}{q}}\,\left(\int_{Q_{R_0}} (u+t)^q\,dx\right)^\frac{1}{q}.
\]
In particular, by recalling that $u\ge 0$ and using Minkowski's inequality, we get
\[
\|u\|_{L^\infty(Q_{R_0/2})}\le 2^\frac{N}{q}\,\mathcal{A}_{N,p,q}\,\mathcal{B}_{N,p,q}\,\Big(2^{p}\,p\,T^p_{N,p}\Big)^\frac{N}{p\,q}\,\left[\left(\fint_{Q_R} |u|^q\,dx\right)^\frac{1}{q}+t\right].
\]
By recalling the choice \eqref{t} of $t$, we conclude. 
\par
Finally, we observe that for every fixed $1< q<N/(N-1)$, we can take $1<p<2$ such that $p\le q$. In particular, since $p>1$, we have $q<p^*$. Moreover, it holds
\[
\lim_{p\searrow 1}2^\frac{N}{q}\,\mathcal{A}_{N,p,q}\,\mathcal{B}_{N,p,q}\,\Big(2^p\,p\,T^p_{N,p}\Big)^\frac{N}{p\,q}=\mathcal{B}_{N,1,q}\,\left(\frac{4}{N\,\omega_N^\frac{1}{N}}\right)^\frac{N}{q}.
\]
In the last identity, we used the definitions of $\mathcal{A}_{N,p,q}$, $\mathcal{B}_{N,p,q}$ and \eqref{costante_talenti}.
\end{proof}
\begin{remark}
We observe that the constant $\mathcal{B}_{N,1,q}$ is given by
\[
\mathcal{B}_{N,1,q}=\exp\left(\frac{1}{q}\sum_{i=0}^\infty (i+2)\,\left(1-\frac{1}{N}\right)^i\,\log 2\right).
\]
This has the following asymptotic behaviour 
\[
\mathcal{B}_{N,1,q}\sim 2^\frac{N^2}{q},
\]
as the dimension $N$ goes to $\infty$.
\end{remark}

\section{A simple approximation lemma for convex sets}
\label{sec:B}

\begin{lemma}
\label{lm:approximation}
Let $N\ge 2$ and let $E\subseteq \mathbb{R}^N$ be an open bounded convex set. Then there exists a sequence of smooth open bounded convex sets $\{E_n\}_{n\ge n_0}\subseteq\mathbb{R}^N$ and a constant $C_E>0$, such that
\begin{equation}
\label{contiene}
\left(1-\frac{C_E}{n}\right)\,E\subseteq E_n\subseteq \left(1+\frac{C_E}{n}\right)\,E,\qquad \mbox{ for every } n\ge n_0.
\end{equation}
Moreover, we have 
\[
\lim_{n\to\infty} \mathcal{H}^{N-1}(\partial E_n)=\lim_{n\to\infty} \mathcal{H}^{N-1}(\partial E).
\]
\end{lemma}
\begin{proof}
We first show that the last property follows from \eqref{contiene}. Indeed, since all the sets involved are convex, by using the monotonicity of the $\mathcal{H}^{N-1}$ measure of the boundaries with respect to set inclusion (see \cite[Lemma 2.2.2]{BuBu}), we have 
\[
\left(1-\frac{C_E}{n}\right)^{N-1}\,\mathcal{H}^{N-1}(\partial E)\le \mathcal{H}^{N-1}(\partial E_n)\le \left(1+\frac{C_E}{n}\right)^{N-1}\,\mathcal{H}^{N-1}(\partial E),\qquad \mbox{ for every } n\ge n_0.
\]
By taking the limit as $n$ goes to $\infty$, we get the desired conclusion.
\par
In order to construct the sequence $\{E_n\}_{n\in\mathbb{N}}$, we suppose for simplicity that $0\in E$ and introduce the {\it Minkowski functional} of $E$, i.e.
\[
j(x)=\inf\Big\{\lambda>0\, :\, x\in \lambda\, E\Big\}.
\]
This is a positively $1-$homogeneous convex function, which is globally Lipschitz continuous and such that for every $\ell>0$ we have
\[
j(x)<\ell\qquad \mbox{ if and only if }\qquad x\in \ell\,E.
\]
Moreover, we have $j(x)\ge 0$ for every $x\in\mathbb{R}^N$ and $j(x)=0$ if and only if $x=0$.
We then set $j_n=j\ast \varrho_n$, where $\{\varrho_n\}_{n\ge 1}$ is the usual family of standard mollifiers. Observe that this is a non-negative smooth convex function. 
We define 
\[
E_n=\Big\{x\in\mathbb{R}^N\, :\, j_n(x)<1\Big\}.
\]
Let us call $C_E$ the Lipschitz constant of $j$. Then we take $x\in (1-C_E/n)\,E$, so that 
\[
j(x)<\left(1-\frac{C_E}{n}\right).
\]
By using the definition of $j_n$, we have 
\[
\begin{split}
j_n(x)=\int_{B_\frac{1}{n}(0)} j(x-y)\,\varrho_n(y)\,dy&=\int_{B_\frac{1}{n}(0)} [j(x-y)-j(x)]\,\varrho_n(y)\,dy+j(x)\\
&\le \frac{C_E}{n}+j(x)<\frac{C_E}{n}+\left(1-\frac{C_E}{n}\right)=1.
\end{split}
\]
This shows the validity of the leftmost inclusion in \eqref{contiene}. In a similar way, if $x\in E_n$ then we have
\[
\begin{split}
1>j_n(x)=\int_{B_\frac{1}{n}(0)} [j(x-y)-j(x)\,\varrho_n(y)]\,dy+ j(x)\ge -\frac{C_E}{n}+j(x),
\end{split}
\]
that is $j(x)<1+C_E/n$, which shows the validity of the rightmost inequality in \eqref{contiene}, as well. 
\par
We are left with observing that $\partial E_n$ coincides with the level line $\{x\in\mathbb{R}^N\, :\, j_n(x)=1\}$. Such a level line is smooth, since it does not contain critical points of $j_n$: indeed, we recall that for a convex function every critical point is automatically a global minimum point. On the other hand, by using the Lipschitz character of $j$ and the fact that $j(0)=0$, we have
\[
j_n(0)=\int_{B_\frac{1}{n}(0)} j(-y)\,\varrho_n(y)\,dy\le \frac{C_E}{n},
\] 
and the latter is strictly less than $1$, for $n$ large enough. This shows that the value $1$ can not be the minimum of $j_n$ and thus $\partial E_n=\{x\in\mathbb{R}^N\, :\, j_n(x)=1\}$ does not contain any critical value of $j_n$.
\end{proof}



\end{document}